\documentclass[a4paper, 12pt, final]{article}  % final
\usepackage{mathtools, amsthm, amssymb, nicefrac} % mathtools loads amsmath
\mathtoolsset{showonlyrefs}

\usepackage[verbose, a4paper, hcentering, marginparwidth=2cm]{geometry} %, total={390pt, 592pt}
\usepackage{xcolor}% http://ctan.org/pkg/{hyperref,xcolor}
\usepackage[unicode=true, pdftitle={Greta et al.}, pdfauthor={Greta, Alois, Jan-Frederik}, colorlinks=true, backref=page, citecolor=blue, linkcolor=gray]{hyperref}
\usepackage[notcite,notref]{showkeys}	%  display labels and bookmarks
\usepackage[textsize= tiny, obeyFinal]{todonotes}  %  enable

\usepackage[numbers]{natbib}
\usepackage{newtxtext, newtxmath}%	replaces Springer like postscript (PS) fonts (Times)
\usepackage{dsfont}             % für \one und 
\usepackage{microtype}			%
\usepackage{doi}
\usepackage[inline]{enumitem}           % customized list environments, as roman
\setlist[enumerate,1]{label=(\roman*)}  % 1st layer (i)
\newtheorem{theorem}{Theorem}
\newtheorem{remark}[theorem]{Remark}
\newtheorem{lemma}[theorem]{Lemma}

\newtheorem{corollary}[theorem]{Corollary}
\newtheorem{definition}[theorem]{Definition}

\DeclareMathOperator{\E}{\mathds E}	        % Expectation
\DeclareMathOperator*{\divergenz}{div}              %
\DeclareMathOperator*{\esssup}{ess\,sup}

\newcommand{\R}{\mathbb{R}}
\newcommand{\RN}{\mathbb{R}^N}
\newcommand{\hu}{\hat{u}}
\newcommand{\halpha}{\hat{\alpha}}
\newcommand{\hbeta}{\hat{\beta}}
\newcommand{\hV}{\hat{V}}
\newcommand{\ha}{\hat{a}}
\newcommand{\hb}{\hat{b}}
\newcommand{\Gammain}{\Gamma_{\text{in}}}
\newcommand{\Gammaout}{\Gamma_{\text{out}}}
\newcommand{\F}{\mathcal{F}}
\newcommand{\B}{\mathcal{B}}

\newcommand{\eps}{\varepsilon}

\newcommand{\Om}{D}
\newcommand{\rand}{\partial D}

\newcommand{\into}{\int_{D}}

\newcommand{\intor}{\int_{\partial\Omega}}

\newcommand{\close}{\overline{\Omega}}

\renewcommand{\l}{\left}
\renewcommand{\r}{\right}

\numberwithin{theorem}{section}
\numberwithin{equation}{section}

\newcommand{\poi}{\Phi}
\newcommand{\W}{d}

\newcommand{\jan}[1]{\todo[inline,color=green!40]{Jan: #1}}

\title{Uncertainty Analysis for Drift-Diffusion Equations}
\author{
	Greta Marino%
	\thanks{University of Technology, Chemnitz, Germany.}\,~\footnote{Corresponding author: \href{mailto:greta.marino@mathematik.tu-chemnitz.de}{greta.marino@mathematik.tu-chemnitz.de}}
	\and
	Jan-Frederik Pietschmann\footnotemark[1]
	\and
	Alois Pichler\footnotemark[1]\,~\thanks{Funded by Deutsche Forschungsgemeinschaft (DFG, German Research Foundation)~-- Project-ID 416228727~-- SFB~1410.}}
\begin{document}
	\maketitle
	\begin{abstract}
		We study evolution equations of drift-diffusion type when various parameters are random. Motivated by applications in pedestrian dynamics, we focus on the case when the total mass is, due to boundary or reaction terms, not conserved.
		After providing existence and stability for the deterministic problem, we consider uncertainty in the data. Instead of a sensitivity analysis we propose to measure functionals of the solution, so-called quantities of interest (QoI), by involving scalarizing statistics.
		For these summarizing statistics we provide probabilistic continuity results.
	\end{abstract}
	\maketitle

%			Introduction
%	╰───────────────────────────────────
\section{Introduction}
	Evolution equations describe a variety of systems in various scientific areas. We focus on non-linear drift-diffusion equations that describe the evolution of an unknown density $u=u(t,x)$. A classical example is the linear Fokker--Planck equation in a bounded domain $D \subset \RN$, $N= 1, 2, 3$, which reads as 
	\begin{align*}
		\partial_t u+ \nabla \cdot (-\nabla u+ u \nabla V) = 0 \text{ in } (0,T) \times D,
	\end{align*}
where $T>0$ is a final time and $V= V(t, x) $ a given external potential. The equation needs to be supplemented with suitable initial- and boundary conditions, and we emphasize that for no flux boundary conditions the mass of the initial data is conserved over time.

Here, we are interested in models in which the mass can change, either due to non-homogeneous boundary- or due to reaction terms. Furthermore, we replace the convection term  $u\nabla V$ by $f(u) \nabla V$ where $f$ is a positive function that degenerates at both $u=0$ and $u=1$. This choice ensures that the solutions satisfy the box constraints $0 \le u \le 1$ a.e. in $(0,T) \times D$ and is often called \emph{volume filling}. Both choices are motivated by applications to pedestrian dynamics where we assume that people may enter or leave the domain and that there is a maximal density (normalized to one) that they can occupy.
%\jan{Hier mehr?}

For concreteness, consider the following two problems. %we outline deterministically first. 
Denoting by 
\[	J \coloneqq -\nabla u+ f(u) \nabla V(t, x)\]
the flux density, the first problem reads 
\begin{equation}\label{problem1}\tag{P1}
	\partial_t u+ \nabla \cdot J = \alpha(t, x)\, f(u)- \beta(t, x)\, u\, \qquad \text{in } (0, T) \times D,
\end{equation}
where $\alpha, \beta\colon (0,T) \times D \to \R_+$ are given functions 
controlling the creation and removal of mass.
%
%supplemented with suitable initial  and boundary conditions, 
%
The second problem is given by
\begin{equation}\label{problem2}\tag{P2}
	\begin{aligned}
		\partial_t u + \nabla \cdot J & = 0 && \text{in }  (0,T) \times D, \\
		- J \cdot n&= a(t, x)\, g(u) && \text{on }  (0,T) \times \Gamma_\text{in}, \\
		J \cdot n&= b(t, x)\, u && \text{on }  (0, T) \times \Gamma_\text{out}, \\
		J \cdot n&= 0 && \text{on }  (0,T) \times \partial D \setminus  (\Gamma_\text{in} \cup \Gamma_\text{out})
	\end{aligned}
\end{equation}
with~$n$ being the outward normal, $a\colon(0,T) \times \Gamma_\text{in} \to \R$, $b\colon(0,T) \times \Gamma_\text{out}\to \R$ and where~$\Gammain$ and~$\Gammaout $ are suitable subsets of $ \partial D$, called in- and outflow parts of the boundary. 
%and $f$, $g$, $\alpha$, $\beta$, $a$, $b$ are smooth functions with properties specified later.

The models~\eqref{problem1} and~\eqref{problem2} appear in various applications, e.g., transport phenomena in biological systems, \cite{humpert_role_2019} or pedestrian dynamics, \cite{GSW}. Here we focus on the latter situation as a guiding example. Then, $u(t,x)$ denotes the (scaled) density of pedestrians at point $x \in D$ at time $t \in (0,T)$ and $u=1$ is the maximal possible density, i.e., the maximal number of people that can be located in a given area. From a practical point of view, it is desirable to avoid very high densities as they can lead to panic situations. Thus, one objective of our work is to study the dependence of these regions on random parameters. In both models, we allow for a change in the total number of pedestrians, either via reaction terms in \eqref{problem1} or via boundary terms in \eqref{problem2}. A prototypical set-up for \eqref{problem2} would be a corridor where people can enter at one side (with rate $\alpha$), walk through and leave at the other side with velocity $\beta$. The function $g(u)$ ensures that entering the domain is only possible if space is available, a typical choice could be $g(u) = 1 - u$. 

Our analysis addresses quantities, which are derived from the entire solution of the differential equations. These functions, which transform the solution function to a single real number, are termed \emph{quantity of interest (QoI)}. 
Important examples include the average pedestrian density within a given area and a specific time interval.
%The QoI naturally correspond to the average density of pedestrians within a given region and within a given time observed, say.

The main result of our work is a \emph{probabilistic} analysis of the QoI on random parameters that appear in the problem.
The boundary conditions to be introduced later as well as the potential reveal over time, driving the evolution of the pedestrian density. Partial observations, as well as the state of the system at later times entail additional information. We employ this additional information for gradually more precise forecasts of the~QoI.

%By assessing the quantity of interest we do not build on exact information available.
To allow for concrete conclusions, we assume that these governing functions driving the evolution of the system realize with unknown outcomes but known law. Summarizing statistics as the expected value aggregate all possible values to a single real number.

Involving the entire law of all possible realizations, however, is expensive.
It is a well-accepted practice in applied economics to replace a possibly complicated measure by a simpler measure, as the quantities of interest depend explicitly on the law. Indeed, the Wasserstein distance provides measurements, which allows assessing these deviations explicitly.

%The application addressed here builds on the evolution of a system. This evolution has a very specific characteristic, is given by the \emph{nature of time}.

For our applications, the information available grows gradually with time. Indeed, at a given instant of time one may observe the realization of the governing functions until now, but the future remains unknown. Waiting an additional time interval will reveal additional information allowing for improved forecasts. %but to know everything one has to wait until the very end of the evolution.

The nested distance relates the known information from the past with the unknown realizations from the future. The nested distance builds on the Wasserstein distance, or is a generalization of the Wasserstein distance, \cite{Ina2021_PhD, RachevRueschendorf}, adapted to the underlying stochastic process.
The nested distance was introduced in \cite{Pflug2009}, see also \cite{PflugPichler2011} in a discrete-time setting,
its topological properties are studied in \cite{BackhoffBeiglbock} and \cite{Acciaio2016} study extensions to continuous time.
This approach will allow us to evaluate the quantities of interest for simple probability measures, which are computationally accessible faster.

%	Control theory, in a random framework, typically involves the expectations in the objective.

%In real world situations, the governing components which drive the solutions are not available with full precision. Thus, we consider various driving parameters as random, namely the functions $\alpha, \beta, a, b$ the potential $V$ but also the initial datum.\\
%
%We focus on functionals of the (then random) solution, called quantity of interest (QoI) 
%Examples could be the total mass at a given time in a subdomain or the size of the subset of the domain $D$ on which the density exceeds a given value, but many more are possible depending on the given application. 

 %To account for this fact, 
 %and, instead of a plain sensitivity analysis, we propose distances, which allow controlling the system throughout and control actions during the system's runtime. \jan{What does that mean?}

%\subsection{Relevant literature}
%
%\subsection{Our contribution}

\jan{TODO: Literatur}

\paragraph{Outline.}
The paper is organized as follows: In Section~\ref{sec:prelim} we introduce the 
precise mathematical setting and notations. Section~\ref{sect3} then discusses 
existence and stability properties of the solutions for the deterministic 
problems with respect to parameters. Then, Section~\ref{sec:random_data} 
addresses the case when parameters become random and shows our main stability 
result. We conclude by a summary in Section~\ref{sec:summary}.

%\begin{equation}\label{problem1}
%	\partial_t u+ \nabla \cdot (-\nabla u+ f(u) \nabla V(t, x))= \alpha(t, x)\, f(u)- \beta(t, x)\, u\, \qquad \text{in } (0, T) \times D,
%\end{equation}
%where $V= V(t, x) $ is a given external potential, $\alpha, \beta : (0,T) \times \Omega \to \R_+$ are given functions controlling the creation and removal of mass.

%*************************************************

\section{Preliminaries}\label{sec:prelim}

Let us first fix the setting: $D \subset \RN$, $N= 1, 2, 3$, denotes an open and 
bounded domain with Lipschitz boundary $\partial D$ and $T > 0$ a fixed final 
time.

Let $r \in [1, \infty)$. We denote by $L^r(D)$, $L^r(D; \RN)$ and $W^{1, r}(D)$ the usual Lebesgue and Sobolev spaces equipped with the norms $\|\cdot\|_{L^r(D)}$ and $\|\cdot\|_{W^{1, r}(D)}$ respectively, given by
\begin{gather}
	\|u\|_{L^r(D)}= \l(\into |u|^r \; dx \r)^{1/r}, \quad \|\nabla u\|_{L^r(D)}= \l(\into |\nabla u|^r \;dx \r)^{1/r} \text{ and} \\
	\|u\|_{W^{1,r}(D)}= \l(\into |\nabla u|^r \;dx \r)^{1/r}+ \l(\into |u|^r \; dx \r)^{1/r}.
\end{gather}
For $r=2$ we write, as usual, $H^1(D)= W^{1, 2}(D)$. Furthermore, we denote by $H^1(D)^*$ the dual space of $H^1(D)$ consisting of all bounded linear functionals $L\colon H^1(D) \to \R$. For $r= \infty$ the norm of $L^{\infty}(D) $ is given by
\[	\|u\|_{L^\infty(D)}= \esssup_D |u|.	\]
On the boundary $\rand$ we use the $(N-1)$-dimensional Hausdorff (surface) measure denoted by~$\sigma$. Then, we define the Lebesgue spaces $L^s(\rand)$ with $1 \le s \le \infty$ and the corresponding norms are given by
	\[
	\|u\|_{L^s(\partial D)}= \l(\int_{\partial D} |u|^s \; d\sigma\r)^{1/s}, \quad 1 \le s< \infty, \quad \|u\|_{L^{\infty}(\rand)}= \esssup_{\rand} |u|.
	\]
It is well known that there exists a unique linear continuous map $\gamma\colon W^{1,r}(D) \to L^{r_*}(\rand)$, called trace map, such that $\gamma(u)= u|_{\rand}$ for all $u \in W^{1,r}(D) \cap C(\close)$, where $r_*$ is the critical exponent on the boundary given by
\[	r_*=
	\begin{cases}
		\frac{(N-1)r}{N-r} & \text{if } r< N, \\
		\infty & \text{if } r \ge N,
	\end{cases}
	\]
see \cite{Adams} for further details. For the sake of notational simplicity we drop the use of the trace map $\gamma$, being understood that all restrictions of the Sobolev functions $u \in W^{1, r}(D)$ on the boundary $\rand$ are defined in the sense of traces. In particular, there exists a constant $c_{\rand}>0$ such that
	\begin{equation}\label{SobolevBound}
		\|u\|_{L^q(\rand)} \le c_{\rand}\, \|u\|_{W^{1, r}(D)}, \quad u \in W^{1,r}(D), \ r \le q \le r_*,
	\end{equation}
see  again \cite{Adams}.

Let $B$ be a Banach space with norm $\|\cdot\|_B$. For every $1 \le r < \infty$ we denote by $L^r((0,T); B)$ the Bochner space of all measurable functions $u\colon [0,T] \to B$ such that
\[
	\|u\|_{L^r((0,T); B)}\coloneqq \l(\int_0^T \|u(t)\|_B^r \; dt \r)^{1/r}< \infty.
\]
For $r= \infty$ the norm of the corresponding space $L^{\infty}((0, T); B)$ is given by
	\[
	\|u\|_{L^{\infty}((0,T); B)}\coloneqq \esssup_{0 \le t \le T} \|u(t)\|_B.
	\]
Finally, the space $C([0,T]; B)$ contains all continuous functions $u\colon [0,T] \to B$ such that
	\[
	\|u\|_{C([0,T]; B)}\coloneqq \max_{0 \le t \le T} \|u(t)\|_B < \infty.
	\]
We refer to \cite{Evans1998} as a reference for the above spaces. Let us recall the following compact embedding, see \cite{Simon}.
\begin{remark}\label{rem:compact}

It holds that
	\[
	L^{\infty}\big((0,T); H^1(D)\big) \cap L^2\big((0,T); H^1(D)^*\big) \hookrightarrow C\big([0,T]; L^2(D)\big),
	\]
compactly.
\end{remark}
%\jan{Norm and Lebesgue measure have same symbol!}
The norm of $\RN$ is denoted by $|\cdot|$ and $\cdot$ denotes the inner product in $\RN$. For every $a \in \R$ we set $a^{\pm}\coloneqq \max\{\pm a, 0\}$ and for $u \in W^{1, r}(D)$ we define $u^{\pm}(\cdot)\coloneqq u(\cdot)^{\pm}$. It is well known that
	\[
	u^\pm \in W^{1, r}(D), \quad |u|= u^++ u^-, \quad u= u^+- u^-.
	\]
Finally, by $|\cdot|$ we denote the Lebesgue measure on $\RN$ and it will be clear from the context which quantity we refer to. 
%\medskip

Let $(\Omega, \mathcal F, P)$ denote a generic probability space. If $V$ is any complete metric space, we let $\mathcal B(V)$  be the $\sigma$-algebra which is generated by the open subsets of $V$, i.e., the Borel $\sigma$-algebra. In other words, $\mathcal B(V)$ is the smallest $\sigma$-algebra containing all open subsets of $V$.
Furthermore, a $V$-valued random variable in $(\Omega, \mathcal F, P)$ is a mapping $X\colon \Omega \to V$ that is $\mathcal F$-$\mathcal B(V)$-measurable, i.e.,
	\[	I \in \B(V) \Rightarrow X^{-1}(I) \in \F.\]
Let $k \ge 1$. For any separable Hilbert space $H$ we denote by $L^k(\Omega, \F, P; H) $ the Bochner space of all random variables $X\colon \Omega \to H$ such that
	\[
	\|X\|_{L^k(\Omega, \F, P; H)}\coloneqq \l(\int_{\Omega} \|X(\omega)\|_H^k \,P(d\omega) \r)^{\nicefrac1k}< \infty.
	\]
In particular, H\"{o}lder's inequality implies the inclusion
	\[
	L^m(\Omega, \F, P; H) \subset L^k(\Omega, \F, P; H) \quad \forall \, m \ge k.
	\]
If $k=2$, then $L^2(\Omega, \F, P; H) $ is a Hilbert space. Abbreviating we shall write $L^2(\Omega; H)$ instead of $L^2(\Omega, \F, P; H) $.

Finally, let $(\mathcal X, d_{\mathcal X})$ and $(\mathcal Y, d_{\mathcal Y})$ be two metric spaces, and let $T\colon \mathcal X \to \mathcal Y$ be an operator. We say that $T$ is $L_T$-Lipschitz continuous if there exists $L_T>0$ such that
\begin{equation}\label{eq:Lipschitz}
	d_{\mathcal Y}\big(T(x_1), T(x_2)\big) \le L_T\, d_{\mathcal X}(x_1, x_2) \quad \text{for all } x_1, x_2 \in \mathcal X.
\end{equation}

%			Lipschitz continuity
%-------------------------------------------------------------------------
\section{Existence and continuous dependence} \label{sect3}
This section investigates existence of unique weak solutions to both problem 
\eqref{problem1} and \eqref{problem2}. Then, the Lipschitz continuity of the 
solution operators for both problems is shown.

%			Lipschitz continuity
%-------------------------------------------------------------------------
\subsection{Problem (\ref{problem1})}

We assume the following hypotheses.
	\begin{enumerate}[label=($H_\arabic*$), ref=($H_\arabic*$)] % , noitemsep, nolistsep
		\item \label{H1}
		The initial condition $u^0 \in W^{2- 2/p, p}(D)$ for some fixed $2< p< 3$ and satisfies $0 \le u^0 \le 1$.
		
		\item\label{H2} The potential $V \in L^{\infty}\big((0,T); W^{1, \infty}(D)\big)$.
		
		\item\label{H3} The functions $\alpha$, $\beta \in L^{\infty}\big((0,T) \times D\big)$ and there exist $\alpha_0, \, \beta_0> 0$ such that
			\[
			\alpha(t, x) \ge \alpha_0 \quad \text{and} \quad \beta(t, x) \ge \beta_0 \quad \text{ a.e.\ in }  (0,T) \times D.
			\]
		\item\label{H4} The function $f \in W^{1, \infty}(\R)$ with Lipschitz constant $L_f$ and is such that $f(0)= f(1)= 0$ and $f(s) \ge 0 $ for every $0< s< 1$.
	\end{enumerate}

We use the following standard notion of weak solution
\begin{definition}\label{eq:weak_P1}
We say that a function $u \in L^2\big((0,T); H^1(D)\big)$ with $\partial_t u \in L^2\big((0,T); H^1(D)^*\big)$ is a weak solution to equation~\eqref{problem1} supplemented with the boundary condition
	\begin{equation}
	\label{boundary}
	(-\nabla u+ f(u) \nabla V) \cdot n= 0 \quad \text{on }  (0,T) \times \rand
	\end{equation}
if the identity
	\begin{equation} \label{weak-sol1}
	\begin{split}
	&\into \partial_t u \varphi \; dx\\
	& \, \,  - \into (-\nabla u+ f(u) \nabla V) \cdot \nabla \varphi \; dx+ \into \beta(t,x) u \varphi \,dx- \into \alpha(t,x) f(u) \varphi \, dx= 0
	\end{split}
	\end{equation}
holds for all $\varphi \in H^1(D)$ and a.e.\ $0 \le t \le T$.
\end{definition}
The following existence result holds. 

\begin{theorem}
\label{rem:existence1}

Let hypotheses \ref{H1}--\ref{H4} hold. Then there exists a unique weak solution to problem \eqref{problem1} in the sense of Definition \ref{eq:weak_P1}. Moreover, we have $0 \le u \le 1$ a.e. in $(0,T) \times D$ as well as $u \in C([0,T] \times \overline D)$ whenever $p>2$ in \ref{H1} holds.
\end{theorem}
The proof follows from a standard application of Banach's fixed point theorem. The bounds $0 \le u \le 1$ are then shown using an appropriate test function with a Gronwall argument. For the convenience of the reader, the proof is shown in Appendix \ref{sec:appendix}.

We next denote by 
	\begin{align}\label{solution1}
	\MoveEqLeft[3] \mathcal S_1\colon L^{\infty}\big((0,T) \times D\big) \times L^{\infty}\big((0,T) \times D\big) \times L^{\infty}\big((0,T); W^{1, \infty}(D)\big) \times L^2(D) \\
	& \qquad \to L^2\big((0, T); H^1(D)\big) \cap L^2\big((0,T); H^1(D)^*\big)
	\end{align}
 the operator that associates to each $(\alpha, \beta, V, u^0) \in L^{\infty}((0,T) \times D) \times L^{\infty}((0,T) \times D) \times L^{\infty}((0,T); W^{1, \infty}(D)) \times L^2(D) $ the solution $u \in L^2((0, T); H^1(D)) \cap L^2((0,T); H^1(D)^*) $ to equation~\eqref{problem1} supplemented with boundary condition~\eqref{boundary}.
Thanks to Theorem \ref{rem:existence1} $\mathcal S_1$ is well-defined. We now show that is it also Lipschitz continuous.
\begin{theorem}[Lipschitz continuity of the solution operator~\(\mathcal S_1\)] \label{thm:Lip1}
	Let $D \subset \RN$ be a bounded domain with Lipschitz boundary $\partial D$ and let hypotheses~\ref{H1}--\ref{H4} be satisfied. Let $u_1$, $u_2$ be two solutions to problem~\eqref{problem1} and~\eqref{boundary} for given data $\alpha_i$, $\beta_i \in L^{\infty}((0,T) \times D), V_i \in L^{\infty}((0,T); W^{1, \infty}(D))$ and initial conditions $u_i^0 \in W^{2-2/p, p}(D)$, $i= 1, 2$. Then, there exist  constants $C, \tilde C>0$ such that
	\begin{align}	\label{ineq}
		\MoveEqLeft[1]\|u_1- u_2\|_{L^2((0,T); H^1(D))}^2 \\
		& \le C e^{\tilde C T} \biggl(\|\alpha_1- \alpha_2\|_{L^{\infty}(D \times (0,T))}^2+ \|\beta_1- \beta_2\|_{L^{\infty}(D \times (0,T))}^2 \\
		& \qquad \qquad + \|V_1- V_2\|_{L^{\infty}((0,T); W^{1, \infty}(D))}^2+  \|u_1^0- u_2^0\|_{L^2(D)}^2 \biggr).\nonumber
	\end{align}
\end{theorem}

\begin{proof}
	Set $\hu\coloneqq u_1- u_2$ and similarly for $\halpha$, $\hbeta$, and $\hV$. Then we consider equation~\eqref{problem1} for $u_1$ and $u_2$, respectively, and take the difference of the corresponding equations. After straightforward calculations the resulting equation reads as
	\begin{align} \label{eq-1}
		\MoveEqLeft[3]\partial_t \hu- \Delta \hu= -\nabla \cdot \bigl((f(u_1)- f(u_2)) \nabla V_1+ f(u_2) \nabla \hV \bigr)  \\
		& \qquad + \halpha f(u_1)+ \alpha_2 \l(f(u_1)- f(u_2)\r)- \hbeta u_1- \beta_2 \hu.\nonumber
	\end{align}
	Multiplying~\eqref{eq-1} by $\hu$, integrating in $D$, and taking the boundary condition \eqref{boundary} into account gives
	\begin{align}\label{eq-2}
		\MoveEqLeft[12]\frac12 \frac{d}{dt} \int_D \hu^2 \; dx+ \into |\nabla \hu|^2 \;dx= \int_D \bigl((f(u_1)- f(u_2)) \nabla V_1+ f(u_2) \nabla \hV \bigr)  \cdot \nabla \hu \;dx\nonumber \\
		& \qquad + \int_D \l(\halpha f(u_1)+ \alpha_2 \l(f(u_1)- f(u_2)\r)\r) \hu \;dx \\
		& \qquad+ \int_D \big(-\hbeta u_1- \beta_2 \hu\big) \hu \; dx.\nonumber
	\end{align}
	We estimate the terms on the right-hand side of~\eqref{eq-2} separately. First, we use the weighted Young's inequality along with  the hypothesis on $f$, see~\ref{H4},  to achieve
		\[
		\begin{split}
		\int_D (f(u_1)- f(u_2)) \nabla V_1 \cdot \nabla \hu \; dx &\le 4 L_f^2 \|\nabla V_1\|_{L^{\infty}(D)}^2 \|\hu\|_{L^2(D)}^2+ \frac14 \|\nabla \hu\|_{L^2(D)}^2
		\end{split}
		\]
	as well as
		\[
		\into f(u_2) \nabla \hV \cdot \nabla \hu  \; dx \le 4 |D| \|f(u_2)\|_{L^{\infty}(D)}^2 \|\nabla \hV\|_{L^{\infty}(D)}^2+ \frac{1}{4} \|\nabla \hu\|_{L^2(D)}^2.
		\]
	Moreover, we use Young's inequality and hypothesis~\ref{H3} to have
		\[
		\begin{split}
		\int_D \halpha f(u_1) \hu \;dx &\le \frac{1}{2} \into  \halpha^2 f(u_1)^2 \; dx+ \frac{1}{2} \into \hu^2 \; dx \\
		&\le \frac{|D|}{2} \|f(u_1)\|_{L^{\infty}(D)}^2 \|\halpha\|_{L^{\infty}(D)}^2+ \frac{1}{2} \|\hu\|_{L^2(D)}^2,
		\end{split}
		\]
	while a simple calculation gives
		\[
		\int_D \alpha_2 \l(f(u_1)-f(u_2)\r) \hu \;dx \le L_f \|\alpha_2 \|_{L^{\infty}(D)} \|\hu\|_{L^2(D)}^2.
		\]
Finally, we use the fact that $u_1 \le 1$ along with Young's inequality and hypothesis~\ref{H2}  to have
		\[
		\int_D -\hbeta u_1  \hu \;dx \le \frac{|D|}{2}  \|\hbeta\|_{L^{\infty}(D)}^2+ \frac12 \|\hu\|_{L^2(D)}^2,
		\]
	as well as  
		\[
		\int_D -\beta_2 \hu^2 \;dx \le \|\beta_2\|_{L^{\infty}(D)}  \|\hu\|_{L^2(D)}^2.
		\]
	Collecting these estimates, equation~\eqref{eq-2} simplifies to
	\begin{align*}
		&\frac12 \frac{d}{dt} \|\hu\|_{L^2(D)}^2+ \frac12 \| \nabla \hu \|_{L^2(D)}^2  \\
		 & \le \l(4 L_f^2 \|\nabla V_1\|_{L^{\infty}(D)}^2+ 1+ L_f \|\alpha_2\|_{L^{\infty}(D)}^2+ \|\beta_2\|_{L^{\infty}(D)}^2 \r)  \|\hu\|_{L^2(D)}^2 \\
		& \qquad+ \frac{|D|}{2} \|f(u_1)\|_{L^{\infty}(D)}^2 \|\halpha\|_{L^{\infty}(D)}^2+ \frac{|D|}{2}  \|\hbeta\|_{L^{\infty}(D)}^2 + 4 |D| \|f(u_2)\|_{L^{\infty}(D)}^2 \| \nabla \hV\|_{L^{\infty}(D)}^2,
	\end{align*}
that is
	\[
	\begin{split}
	&\frac{d}{dt} \|\hu\|_{L^2(D)}^2+ \|\nabla \hu\|_{L^2(D)}^2 \\
	& \qquad \le B_1 \|\hu\|_{L^2(D)}^2+ B_2 \|\halpha\|_{L^{\infty}(D)}^2+ B_3 \|\hbeta\|_{L^{\infty}(D)}^2+ B_4 \|\hV\|_{W^{1, \infty}(D)}^2,
	\end{split}
	\]
where $B_1= B_1(V_1, \alpha_2, \beta_2, L_f), B_2= B_2(D, \|f\|_{L^\infty(D)}), B_3= B_3(D), B_4= B_4(D, \|f\|_{L^\infty(D)})$. An application of  Gronwall's lemma then yields
	\begin{equation}
	\label{gronwall}
	\begin{split}
	& \|\hu(\cdot \,, t)\|_{L^2(D)}^2+ \int_0^t \|\nabla \hu\|_{L^2(D)}^2 \; ds \\
	& \qquad \le C e^{B_1 T} \l(\|\hV\|_{L^\infty((0,T); W^{1, \infty}(D))}^2+ \|\hbeta\|_{L^{\infty}((0,T) \times D)}^2+ \|\halpha\|_{L^{\infty}((0,T) \times D)}^2+ \|\hu^0\|_{L^2(D)}^2 \r),
	\end{split}
	\end{equation}
for every $t \in [0, T]$, where the constant $C>0$ depends on $T$ but not on the difference $u_1- u_2$. Note that $\hat u^0 \in L^2(D)$ by H\"{o}lder's inequality.
This fact together with Remark \ref{rem:compact} implies 
	\begin{equation}
	\label{ineq-gronw}
	\begin{split}
	& \max_{t \in [0,T]} \|\hu\|_{L^2(D)}^2+  \|\nabla \hu\|_{L^2((0,T); L^2(D))}^2  \\
	& \qquad \le C e^{B_1 T} \l(\|\hV\|_{L^\infty((0,T); W^{1, \infty}(D))}^2+ \|\hbeta\|_{L^{\infty}((0,T) \times D)}^2+ \|\halpha\|_{L^{\infty}((0,T) \times D)}^2+ \|\hu^0\|_{L^2(D)}^2 \r).
	\end{split}
	\end{equation}
Finally, it follows that
		\[
		\begin{split}
		& \|\hu\|_{L^2((0,T); H^1(D))}^2 \\
		& \qquad \le C e^{B_1 T} \l(\|\hV\|_{L^{\infty}((0,T); W^{1, \infty}(D))}^2+  \|\hbeta\|_{L^{\infty}((0,T) \times D)}^2+ \|\halpha\|_{L^{\infty}((0,T) \times D)}^2+   \|\hat u^0\|_{L^2(D)}^2 \r),
		\end{split}
		\]
and then the claim.
\end{proof}

\begin{remark} \label{rem:1}
	Taking the definition of the operator $\mathcal S_1$ into account, Theorem~\ref{thm:Lip1} shows that $\mathcal S_1$ is Lipschitz continuous.
\end{remark}

\begin{remark} \label{rem-bound}
	From equation \eqref{gronwall} we in particular infer that
		\[
		\begin{split}
		& \|\hu(t, \cdot) \|_{L^2(D)}^2 \\
		& \qquad \le C e^{B_1 T} \l(\|\hV\|_{L^{2}((0,T); W^{1, \infty}(D))}^2+  \|\hbeta\|_{L^{\infty}((0,T) \times D)}^2+ \|\halpha\|_{L^{\infty}((0,T) \times D)}^2+   \|\hat u^0\|_{L^2(D)}^2 \r)
		\end{split}
		\]
for every $ t \in [0,T]$, which in particular holds for $t= T$.
\end{remark}

\begin{remark}
	It is easy to check that the properties of a norm imply that $u \mapsto \|u\|_{L^2((0,T); L^2(D))} $ is Lipschitz continuous, because for every $u_1$, $u_2 \in L^2((0,T); L^2(D))$  it holds
	\[
	\l| \|u_1\|_{L^2((0,T); L^2(D))}- \|u_2\|_{L^2((0,T); L^2(D))} \r| \le \|u_1- u_2\|_{L^2((0,T); L^2(D))}.
	\]
\end{remark}

\begin{proof}
	Elementary algebraic calculation. 
\end{proof}

%			Second problem
%-------------------------------------------------------------------------
\subsection{Problem (\ref{problem2})}\label{sect4}

In this section we consider problem~\eqref{problem2} and make the following 
assumptions.
	\begin{enumerate}[label=($H_\arabic*^\prime$), ref=($H_\arabic*^\prime$)] % , noitemsep, nolistsep
		\item\label{H1p} The initial condition $\underline u^0  \in L^2(D)$ and satisfies $0 \le \underline u^0 \le 1$.

		\item\label{H2p} The subsets $\Gamma_{\text{in}}, \Gamma_{\text{out}} \subset \partial D $ of the boundary are open and disjoint. Moreover, $\Gamma_{\text{out}}$ is nonempty.
		
		\item\label{Hf} The function $f \in W^{1, \infty}(\R)$ with Lipschitz constant $L_f$ and is such that $f(0)= f(1)= 0$ and $f(s) \ge 0 $ for every $0< s< 1$.
		
		\item\label{H3p} The function $g \in W^{1, \infty}(\R)$ with Lipschitz constant $L_g$, it is monotonically decreasing and such that $g(0)= 1$ and $g(1)= 0$.
		
		\item\label{H4p} The functions $a \in L^{\infty}((0,T) \times \Gamma_{\text{in}})$ and $b \in L^{\infty}((0,T) \times \Gamma_{\text{out}})$, and there exist $a_0, b_0> 0$ such that
			\[
			a_0 \le a(t, x)\le 1  \quad \text{a.e. in }  (0,T) \times \Gamma_{\text{in}} \quad \text{ and } \quad b_0 \le b(t, x)\le 1 \quad \text{a.e. in }  (0,T) \times \Gamma_{\text{out}}.
			\]
%		Furthermore, it holds $\min\{a, b\}< 1$ and
%			\begin{equation}
%			\label{a-condition}
%			L_g \|a\|_{L^{\infty}(\Gammain)}  \le \frac1{8\, c_{\partial D}^2},
%			\end{equation}
%		as well as
%			\begin{equation}
%			\label{b-condition}
%			\|b\|_{L^{\infty}(\Gamma_{\text{out}})} \le \frac1{8 c_{\partial D}^2}.
%			\end{equation}
		\item\label{H5p} The potential satisfies $V \in L^2((0,T); W^{1, \infty}(D))$ and $\Delta V= 0$.
	\end{enumerate}

\begin{definition}\label{eq:weak_P2}
We say that a function $u \in L^2((0,T); H^1(D)) $ with $\partial_t u \in L^2((0, T); H^1(D)^*)$ is a weak solution to~\eqref{problem2} if
	\begin{equation}
	\label{weak}
	\begin{split}
	& \int_D \partial_t u \varphi \;dx- \int_D(-\nabla u+ f(u) \nabla V) \cdot \nabla \varphi \;dx+ \int_{\Gamma_{\emph{out}}} b(t, x) u \varphi \;d\sigma \\
	& \qquad = \int_{\Gamma_{\emph{in}}} a(t, x) g(u) \varphi \;d\sigma
	\end{split}
	\end{equation}
holds for all $\varphi \in H^1(D)$ and for a.e. $ t \in [0, T]$.
\end{definition}
Combining \cite[Lemma 3.5]{BHP} and \cite[Theorem~1 and Corollary~1]{GSW} we obtain the following existence result. 
\begin{theorem}\label{thm:existence2}
	Let hypotheses~\ref{H1p}--\ref{H5p} hold. Then there exists a unique solution $u \in L^2((0,T); H^1(D)) \cap L^2((0,T); H^1(D)^*)$ such that $0 \le u \le 1$ and $u \in C((0,T) \times \overline D)$.
\end{theorem}
Different from the proof of Theorem \ref{eq:weak_P1}, this proof is based on a fixed point argument in entropy variables defined as derivative of an appropriate entropy functional, see \cite{Ansgar_bounded} for details.

With this result at hand, we again define a parameter-to-solution operator. More precisely, let
\begin{align}\label{solution2}
	\MoveEqLeft[3]\mathcal S_2 \colon L^{\infty}((0,T) \times \Gammain) \times L^{\infty}((0,T) \times \Gammaout) \times L^2((0,T); W^{1, \infty}(D)) \times L^2(D) \\
	& \qquad \to L^2((0,T); H^1(D)) \cap L^2((0,T); H^1(D)^*)
\end{align}
be the operator that maps $(a, b, V, \underline u^0) \in L^{\infty}((0,T) \times \Gammain) \times L^{\infty}((0,T) \times \Gammaout) \times L^2((0,T); W^{1, \infty}(D)) \times L^2(D) $ to the solution to~\eqref{problem2}. Thanks to Theorem \ref{thm:existence2} we see that  it is well-defined.

The next result states the Lipschitz continuity of the operator $\mathcal S_2$.
\begin{theorem}[Lipschitz continuity of the solution operator \(\mathcal S_2\)]\label{thm:Lip2}
	Let $D \subset \R^N$, $N= 1, 2, 3$, be a bounded domain with Lipschitz boundary $\partial D$, and let hypotheses~\ref{H1p}--\ref{H5p} be satisfied. Moreover, let $u_1, u_2$ be two weak solutions to~\eqref{problem2} in the sense of~\eqref{weak} for given data $a_i, b_i, V_i$ and  initial conditions $u_i^0$, $i=1, 2$.
	Then, there exist constants $C, \tilde C>0$, independent of the difference $u_1- u_2$, such that 
	\begin{align}
	\MoveEqLeft[3]\|u_1- u_2\|_{L^2((0,T); H^1(\Om))}^2 \\
	&  \le C e^{\tilde C T} \biggl( \|a_1- a_2\|_{L^{\infty}((0,T) \times \Gamma_{\emph{in}})}^2+  \|b_1- b_2 \|_{L^{\infty}((0,T) \times \Gamma_{\emph{out}})}^2\\
	& \qquad +  \|V_1- V_2\|_{L^2((0,T); W^{1, \infty}(D))}^2+  \|\underline u_1^0- \underline u_2^0\|_{L^2(D)}^2 \biggr).
	\end{align}
\end{theorem}

\begin{proof}
	By $\hu $ we denote the difference $u_1- u_2$, and assume the same meaning for $\ha, \hb$, $\hat V$, and $\underline \hu^0$. Then we write~\eqref{weak} for $u= u_1$ and $u= u_2$ and take the difference of the respective equations. After simplifications the resulting equation reads as
	\begin{align}\label{eq-result}
		\MoveEqLeft[2]\int_D \partial_t \hu \varphi\, dx+ \int_D \nabla \hu \cdot \nabla \varphi \, dx\nonumber \\
		= &  \into \l( \bigl(f(u_1)- f(u_2)\bigr) \nabla V_1+ f(u_2) \nabla \hat V \r) \cdot \nabla \varphi \, dx \\
		&- \int_{\Gammaout} \l(b_1(t,x) \hu+ \hb(t,x) u_2 \r) \varphi \, d\sigma+ \int_{\Gammain} \l(\ha(t,x) g(u_1)+ a_2(t,x) \l(g(u_1)- g(u_2) \r) \r) \varphi \, d\sigma. \nonumber
	\end{align}
	Choosing $\varphi= \hu$ in~\eqref{eq-result}   gives
		\begin{align}\label{eq-result2}
			\MoveEqLeft[3]\frac12 \frac{d}{dt} \int_D  \hu^2 \, dx+ \int_D |\nabla \hu|^2 \, dx\nonumber \\
			=&  \into \l( \bigl(f(u_1)- f(u_2)\bigr) \nabla V_1+ f(u_2) \nabla \hat V \r) \cdot \nabla \hu \, dx \\
			& - \int_{\Gammaout} \l(b_1(t,x) \hu+ \hb(t,x) u_2 \r) \hu \, d\sigma\\
			&+ \int_{\Gammain} \l(\ha(t,x) g(u_1)+ a_2(t,x) \l(g(u_1)- g(u_2) \r) \r) \hu \, d\sigma. \nonumber
		\end{align}
	We aim to estimate the different terms appearing on the right-hand side of~\eqref{eq-result2} separately. First,  hypotheses~\ref{Hf}--\ref{H5p} together with the weighted Young's inequality give
%		\[
%		\begin{split}
%		\into \hu g(u_1) \nabla V_1 \cdot \nabla \hu \, dx & \le 8 \into \hu^2 g(u_1)^2 |\nabla V_1|^2 \, dx+ \frac18 \into |\nabla \hu|^2 \, dx \\
%		& \le 8 \|g(u_1)\|_{L^{\infty}(D)}^2 \|\nabla V_1\|_{L^{\infty}(D)}^2 \|\hu\|_{L^2(D)}^2+ \frac18 \|\nabla \hu\|_{L^2(D)}^2,
%		\end{split}
%		\]
%	as well as
	\begin{align*}
		\MoveEqLeft[4]\int_D \l(f(u_1)- f(u_2) \r) \nabla V_1 \cdot \nabla \hu \, dx \le \into L_f  \hu \nabla V_1 \cdot \nabla \hu dx \\
		& \le 8 L_g^2 \|\nabla V_1 \|_{L^{\infty}(D)}^2 \|\hu\|_{L^2(D)}^2+ \frac18 \|\nabla \hu\|_{L^2(D)}^2,
	\end{align*}
	and
	\begin{align*}
		\MoveEqLeft[4]\into f(u_2) \nabla \hat V \cdot \nabla \hu \, dx \le 8 \into f(u_2)^2 |\nabla \hat V|^2 \, dx+ \frac18 \into |\nabla \hu|^2 \, dx \\
		& \le 8 |D| \|f(u_2)\|_{L^{\infty}(D)}^2 \|\nabla \hV\|_{L^{\infty}(D)}^2+ \frac{1}{8} \|\nabla \hu\|_{L^2(D)}^2.
		\end{align*}
	Moreover, hypothesis~\ref{H4p} together with the weighted Young's inequality and the Sobolev embedding imply
	\begin{align*}
		\MoveEqLeft[5] -\int_{\Gammaout} \hb(t,x) u_2 \hu d\sigma = 
		-\int_{\Gammaout} \l(c_{\partial D} \hb(t,x) u_2 \r) \l(\frac{1}{c_{\partial D}} \hu \r) \, d\sigma \\
		&  \le 8 c_{\rand}^2 \|\hb\|_{L^{\infty}(\Gammaout)}^2 \|u_2\|_{L^2(\rand)}^2+ \frac{1}{8 c_{\rand}^2} \|\hu\|_{L^2(\rand)}^2 \\
		& \le 8 c_{\rand}^2 \|u_2\|_{L^2(\rand)}^2 \|\hb\|_{L^{\infty}(\Gammaout)}^2 + \frac{1}{8} \|\hu\|_{H^{1}(\Om)}^2,
	\end{align*}
	%Condition~\eqref{b-condition} for $b= b_1$ and the Sobolev embedding on the boundary gives
	while the term $-\int_{\Gammaout} b_1(t,x) \hu^2 d\sigma$ can be neglected having a sign. 
%	\begin{align*}
%		\MoveEqLeft[4]-\int_{\Gammaout} b_1(x) \hu^2 d\sigma \le \|b_1\|_{L^{\infty}(\Gammaout)} \|\hu\|_{L^2(\rand)}^2 \\
%		& \le \|b_1\|_{L^{\infty}(\Gammaout)} c_{\rand}^2 \|\hu\|_{H^{1}(\Om)}^2 \\
%		& \le \frac{1}{8}  \|\hu\|_{H^{1}(\Om)}^2.
%	\end{align*}
	From hypothesis~\ref{H3p} we know that $g(u_1) \in L^{\infty}(D)$. Then from \cite[Proposition~2.4]{MW} it follows that $g(u_1) \in L^{\infty}(\rand)$.
	Therefore, we can apply the weighted Young's inequality together with the Sobolev embedding on the boundary to achieve
	\begin{align*}
		\MoveEqLeft[4] \int_{\Gammain} \ha(t,x) g(u_1) \hu \, d\sigma = \int_{\Gammain} \l(c_{\partial D} \ha(t,x) g(u_1)\r) \l(\frac{1}{c_{\partial D}} \hu\r) \, d\sigma \\
		& \le 8 \int_{\Gammain} c_{\partial D}^2 \ha(t,x)^2 g(u_1)^2 \, d\sigma+ \frac{1}{8 c_{\partial D}^2} \intor \hu^2 \, d\sigma \\
		& \le 8 |\partial D| c_{\partial D}^2 \|g(u_1)\|_{L^{\infty}(\partial D)}^2 \|\ha\|_{L^{\infty}(\Gammain)}^2+ \frac{1}{8} \|\hu\|_{H^{1}(D)}^2.
	\end{align*}
	Finally, the monotonicity of $g$ (hypothesis~\ref{H3p}), together with $a_2 \ge 0$, implies
	\[
	\begin{split}
	\int_{\Gammain} a_2(t,x) \big(g(u_1)- g(u_2)\big) \hu \, d\sigma  &\le 0,
%	
%	\int_{\Gammain} a_2(x) L_g \hu^2 \, d\sigma \\
%		& \le \|a_2\|_{L^{\infty}(\Gammain)} L_g \|\hu\|_{L^2(\partial D)}^2 \\
%		& \le \frac{1}{8} \|\hu\|_{H^{1}(D)}^2.
	\end{split}
	\]
	so that this term can also be neglected.
	Taking into account these estimates, equation~\eqref{eq-result2} simplifies to
%	\jan{Greta: Check constants...}
	\begin{align*}
		\MoveEqLeft[4]\frac{1}{2} \frac{d}{dt} \| \hu\|_{L^2(D)}^2 + \frac{1}{8} \|\nabla \hu\|_{L^2(D)}^2  \\
		\le & \l(8 \|f(u_1)\|_{L^{\infty}(D)}^2 \|\nabla V_1\|_{L^{\infty}(\Om)}^2+ 8 L_f^2 \|\nabla V_1\|_{L^{\infty}(D)}^2+ \frac{1}{4} \r) \|\hu\|_{L^2(\Om)}^2 \\
		& + 8 |D| \|f(u_2)\|_{L^{\infty}(\Om)}^2 \|\hat V\|_{W^{1, \infty}(\Om)}^2 \\
		& + 8 c_{\rand}^2 \|u_2\|_{L^2(\rand)}^2 \|\hb\|_{L^{\infty}(\Gammaout)}^2+ 8 |\rand|  c_{\rand}^2 \|f(u_1)\|_{L^{\infty}(\rand)}^2 \|\ha\|_{L^{\infty}(\Gammain)}^2,
	\end{align*}
that is
	\begin{equation} \label{eq-a}
	\begin{split}
	& \frac{d}{dt} \|\hu\|_{L^2(D)}^2+ \|\nabla \hu\|_{L^2(D)}^2 \\
	& \qquad \le A_1 \|\hu\|_{L^2(D)}^2+ A_2 \|\hat V\|_{W^{1, \infty}(D)}^2+ A_3 \|\hb\|_{L^{\infty}(\Gammaout)}^2+ A_4 \|\ha\|_{L^{\infty}(\Gammain)}^2,
	\end{split}
	\end{equation}
where $A_1= A_1(g(u_1), V_1, L_g), A_2= A_2(D, g(u_2)), A_3= A_3(c_{\rand}, u_2), A_4= A_4(\rand, c_{\rand}, g(u_1))$. 

A Gronwall argument then gives
	\[
	\begin{split}
	& \|\hu(t, \cdot)\|_{L^2(D)}^2+ \int_0^t \|\nabla \hu\|_{L^2(D)}^2 \; dx \\
	& \qquad \le C e^{A_1 T} \l(\|\ha\|_{L^{\infty}((0,T) \times \Gammain)}^2+  \|\hb\|_{L^{\infty}((0,T) \times \Gammaout)}^2+ \|\nabla \hat V\|_{L^2((0,T); L^{\infty}(\Om))}^2 + \|\underline \hu^0\|_{L^2(\Om)}^2\r)
	\end{split}
	\]
for every $t \in [0, T]$ and for a constant $C>0$ that does not depend on the difference $u_1- u_2$. Taking into account that $u \in C((0,T) \times \overline D)$, see Theorem \ref{thm:existence2}, gives
		\begin{align*}
			&\max_{t \in [0,T]} \|\hu\|_{L^2(\Om)}^2+ \|\nabla \hu\|_{L^2((0,T); L^2(\Om))}^2 \\
			&\qquad\le C e^{A_1 T} \l(\|\ha\|_{L^{\infty}((0,T) \times \Gammain)}^2+  \|\hb\|_{L^{\infty}((0,T) \times \Gammaout)}^2+ \|\nabla \hat V\|_{L^2((0,T); L^{\infty}(\Om))}^2+  \|\underline\hu^0\|_{L^2(\Om)}^2 \r).
		\end{align*}
It follows that
		\[
		\begin{split}
		& \|\hu\|_{L^2((0,T); H^1(\Om))}^2 \\
		&  \le C e^{A_1 T}  \l(\|\ha\|_{L^{\infty}((0,T) \times \Gammain)}^2+  \|\hb\|_{L^{\infty}((0,T) \times \Gammaout)}^2+ \|\hat V\|_{L^2((0,T); W^{1,\infty}(\Om))}^2+  \|\underline \hu^0\|_{L^2(\Om)}^2 \r),
		\end{split}
		\]
and then the claim.
\end{proof}

\begin{remark}\label{rem:2}
	The previous theorem implies in particular that the operator $\mathcal S_2$ 
	is Lipschitz continuous.
\end{remark}

\subsection{Quantities of interest}\label{sec:QoI}

As mentioned in the introduction, we are eventually interested in understanding the sensitivity of the distribution of functionals of the solution, quantities of interest (QoI), resulting from different random models for the uncertain coefficients $\alpha$, $\beta$, $V$, $u^0$ or $a$, $b$, $V$, $\underline u^0$. 

More precisely, we look at maps \(\poi\colon\Xi\to \mathbb R\), where 
\begin{equation}\label{eq:Xi}
	\Xi\coloneqq L^2\big((0,T); H^1(\Om)\big) \ \cap\ L^2\big((0,T); H^1(\Om)^*\big).
\end{equation}
Candidates for the quantity of interest include the functions
\begin{align}
	&\poi(u)\coloneqq \int_{D^\prime} u(t,x)\,d x\quad\text{or}\quad
	\poi(u)=\int_{t_1}^{t_2}\int_{D} u(t,x)\,d x,
\end{align}
with $D' \subset D$ and $t_1<t_2\le T$.
In terms of our application of pedestrian motion, they can be interpreted as the amount of pedestrians in the subdomain $D^\prime$, or the total amount in the fixed time windows $[t_1;\,t_2]$.
Note that both examples are Lipschitz continuous which will be crucial later on.

Another interesting example is the number of pedestrians in regions that exceed a given density. The reason is that in these regions, small variations are sufficient to reach a density so high that, e.g., panic may occur. This corresponds to the measure of superlevel sets, i.e.,
\begin{equation}
	\label{super-level}
	\poi(u)\coloneqq |\{x \in D\colon u(t, x)> c\}|,
\end{equation}
where $0 \le c \le 1$ is a given constant. The issue here is that $\poi(u)$ in~\eqref{super-level} is no longer Lipschitz continuous w.r.t.~$u$ which will be an essential assumption for our subsequent analysis. Thus, we introduce a regularized version as follows. First observe that 
\[
	\poi(u)= \int_{\{x \in D: \, u(t, x)>c\}} dx= \into \chi_{\{u(t,x)>c\}}(x)\,  dx.
\]
Now we replace $\chi_{\{u(t,x)>c\}}(x)$ by $\chi_{c}^\eps(u(t,x))$ which is a smooth (and thus Lipschitz continuous) approximation %. %\jan{todoo...}
%Moreover, it is known that there exists a sequence of $L_{\eps}$-Lipschitz 
%functions $\l(\chi^{\eps}(t)\r)_{\eps>0} \subset C^{\infty}(D_{\eps})$ that 
%Tihs alapproximate $\chi_{\{u>c\}}(t)$, that is,
and define 
\begin{align*}
	\poi(u) \coloneqq \into \chi_c^\eps(u(t,x))\,dx.
\end{align*}
\section{Random data}\label{sec:random_data}
The aim of this section is to extend the results of the previous section to the random setting.
A traditional way of solving an evolution equation as exemplified in~\eqref{problem1} is by involving the Feynman--Kac representation.
That is, by introducing a stochastic process \(X_t\), \(t\in(0,T)\), driven by a Brownian motion so that the solution of the differential equation can be given by a conditional expectation %as
% \begin{equation}\label{eq:19}
% 	u(x,t)= \E\left[\left. \int_t^T f(X_s, s) \exp\left(-\int_t^T\beta(X_s)\, e^{-V(X_s,s)} \mathrm ds\right)\right|X_t=x \right]
% \end{equation}
(cf.\ \cite{Karatzas1988}). %Here, the expectation is with respect to the measure of the Wiener process.

%\todo{Alois: was passt nicht?}
A further, classic approach is given by considering data at random. That is, parameters or functions governing the differential equation are chosen randomly so that the solution of the differential equation, as a consequence, is random as well. However, the specific differential equations~\eqref{problem1} and~\eqref{problem2} involve the additional and novel aspect of \emph{time}.
Here, the parameters are typically not revealed at once but gradually over time, which is a distinctive element of the equations~\eqref{problem1} and~\eqref{problem2} compared to other traditional and classical settings just mentioned. It is our goal here to extend \cite[Theorem~17, Theorem~22 and Corollary~23]{ErnstPichlerSprungk2020} to incorporate the new dimension time and the new aspect of gradually unveiled knowledge of parameters. 
%The following sections elaborate on this aspect.

The limited knowledge about uncertain data is modelled probabilistically in the parameters of the differential equation. To elaborate this probabilistic setting we first establish the link to the results of the preceding sections.

Random data are given by considering the function-valued realizations
\begin{align*}
	& \alpha(\cdot \,, \cdot \,, \omega),\ \beta(\cdot \,, \cdot \,, \omega) \in L^{\infty}((0,T) \times D),  && V(\cdot \, , \cdot \,, \omega) \in L^{\infty}((0,T); W^{1, \infty}(\Om)), \\
	& a(\cdot \,, \cdot \,, \omega) \in L^{\infty}((0,T) \times \Gammain), \ && b(\cdot \,, \cdot \,, \omega) \in L^{\infty}((0,T) \times \Gammaout), \\
	& V(\cdot \,, \cdot \,, \omega) \in L^2((0,T); W^{1, \infty}(D)), && u^0(\cdot \,, \omega) \in W^{2-2/p, p}(D),	\\
	& \underline u^0(\cdot \,, \omega) \in L^2(D)
\end{align*}
for $\omega \in \Omega$ with respect to a reference probability space $(\Omega, \mathcal F, P)$. The random data now result in a random solution $u(\cdot \, , \cdot \,, \omega) $ of~\eqref{problem1} and \eqref{problem2}, that is, a random variable $u\colon \Omega \to L^2\big(0,T; H^1(D)\big) \cap L^2\big(0,T; H^1(D)^*\big)$ such that either
\begin{equation}	\label{random-prob1}
	\begin{aligned}
	& \partial_t u(\cdot \,, \cdot \,, \omega)+ \nabla \cdot (-\nabla u(\cdot \,, \cdot \,, \omega)+ f(u(\cdot \,, \cdot \,, \omega)) \nabla V(\cdot \,, \cdot \,, \omega)\\
	& \qquad = \alpha(\cdot \,, \cdot \,, \omega) f(u(\cdot \,, \cdot \,, \omega))- \beta(\cdot \,, \cdot \,, \omega) u(\cdot \,, \cdot \,, \omega)  \quad && \text{in }  (0,T) \times D, \\
	& (-\nabla u(\cdot \,, \cdot \,, \omega)+ f(u(\cdot \,, \cdot \,, \omega)) \nabla V(\cdot \,, \cdot \,, \omega)) \cdot n= 0 && \text{on }  (0,T) \times \rand,
	\end{aligned}
\end{equation}
or
\begin{equation}	\label{random-prob2}
	\begin{aligned}
		\partial_t u(\cdot \, , \cdot \,, \omega)&= \divergenz J(\cdot \,, \cdot \,, \omega)    \quad && \text{in }  (0,T) \times D, \\
		-J(\cdot \,, \cdot \,, \omega) \cdot n&= a(\cdot \,, \cdot \,, \omega) g(u(\cdot \,, \cdot \,, \omega)) && \text{on }  (0,T) \times \Gammain, \\
		J(\cdot \,, \cdot \,, \omega) \cdot n&= b(\cdot \,, \cdot \,, \omega) u(\cdot \,, \cdot \,, \omega) && \text{on }  (0,T) \times \Gammaout, \\
		J(\cdot \,, \cdot \,, \omega) \cdot n&= 0 && \text{on }  (0,T) \times \rand \setminus (\Gammain \cup \Gammaout)
	\end{aligned}
\end{equation}
hold (in a weak sense) $P$~almost surely ($P$\nobreakdash-a.s., for short), 
where $ J(\cdot \,, \cdot \,, \omega)= -\nabla u(\cdot \, , \cdot \, , \omega)+ f(u(\cdot \,, \cdot \,,  \omega)) \nabla V(\cdot \,, \cdot \,, \omega) $.

We say that the stochastic process $\big(u(t)\big)_{t \in [0,T]} $ is a solution to problem~\eqref{random-prob1} if $u \in L^2(0,T; H^1(D)) \cap L^2(0,T; H^1(D)^*)$ \(P\)\nobreakdash-a.s.\ and $\big(u(t)\big)_{t \in [0,T]} $ satisfies the  equation
	\[
	\into \partial_t u \varphi \; dx- \into (-\nabla u+ f(u) \nabla V) \cdot \nabla \varphi \; dx+ \into \beta(t,x) u  \;dx- \into \alpha(t,x) f(u) \varphi \; dx= 0,\quad P\text{-a.s.},
	\]
for all $ \varphi \in H^1(D)$.

We have the following result.
\begin{theorem}\label{thm-random1}
	Let $D \subset \RN$, $N= 1, 2, 3$, be a bounded domain with Lipschitz boundary $\rand$ and let $T> 0$. Suppose that
	\begin{enumerate}[label=(a\textsubscript{\arabic*}), ref=(a\textsubscript{\arabic*})] % , noitemsep, nolistsep
		\item The initial condition satisfies $u^0 \in L^q\big(\Omega; W^{2-2/p, p}(D)\big)$ for some $2< p< 3$ and for all $q \in [1, \infty)$ and the condition $P(\{\omega \in \Omega\colon \, 0\le u^0\le 1\})= 1$.
		
		\item The potential $V \in L^q(\Omega; L^{\infty}((0,T); W^{1, \infty}(D))) $ for all $q \in [1, \infty)$.
		\item The functions $\alpha, \beta \in L^q(\Omega; L^{\infty}(D \times (0,T)))$ for all $q \in [1, \infty)$ and there exist $\alpha_0, \beta_0>0$ such that
			\[
			P(\{\omega \in \Omega\colon \alpha \ge \alpha_0\})= P(\{\omega \in \Omega\colon \beta \ge \beta_0\})= 1.
			\]
		\item The function $f \in W^{1, \infty}(\R)$ with Lipschitz constant $L_f$ and is such that $f(0)=1$ and $f(1)= 0$ and $f(s) \ge 0 $ for every $0< s< 1$.
	\end{enumerate}
	Then, there exists a solution $u \in L^2\big(0,T; H^1(D)\big) \cap L^2\big(0, T; H^1(D)^*\big) $ to problem~\eqref{random-prob1}, $P$\nobreakdash-a.s.
\end{theorem}

\begin{proof}
	Fix $\omega \in \Omega$ and consider the function $u(\cdot \,, \cdot \,, \omega)$. From Theorem \ref{rem:existence1} we know that there exists a solution $u \in L^2(0,T; H^1(D)) \cap L^2(0,T; H^1(D)^*)$ to problem~\eqref{problem1}--\eqref{boundary}, that is, $u$ is a solution to~\eqref{random-prob1}, $P$\nobreakdash-a.s.
\end{proof}

In a very similar way, we say that the stochastic process $\l(u(t)\r)_{t \in [0,T]} $ is a solution to problem~\eqref{random-prob2} if $u \in L^2(0,T; H^1(D)) \cap L^2(0,T; H^{1}(D)^*)$ a.s.\ $ \omega \in \Omega$ and $\l(u(t)\r)_{t \in [0,T]} $ satisfies the  equation
	\begin{equation*}\label{weak-random}
		\int_D \partial_t u\,\varphi \; dx- \int_D (-\nabla u+ f(u)\,\nabla V) \cdot \nabla \varphi \; dx+ \int_{\Gammaout} b\,u\, \varphi \; d\sigma= \int_{\Gammain} a g(u)\, \varphi \; d\sigma
	\end{equation*}
for all $ \varphi \in H^1(D)$, $P$\nobreakdash-a.s. 

Then we can state the following result.
\begin{theorem}\label{thm-random2}
	Let $D \subset \RN$, $N= 1, 2, 3$, be a bounded domain with Lipschitz boundary $\rand$ and let $T> 0$. Suppose that
	\begin{enumerate}[label=(b\textsubscript{\arabic*}), ref=(b\textsubscript{\arabic*})] % , noitemsep, nolistsep
		\item The initial condition  $\underline u^0 \in L^q(\Omega; L^2(D))$ for every $q \in [1, \infty)$ and satisfies
			\[
			P\big(\{\omega \in \Omega\colon 0 \le \underline u^0  \le 1\}\big)= 1.
			\]
		
		\item The subsets $\Gamma_{\emph{in}}, \Gamma_{\emph{out}} \subset \rand$ are open and disjoint. Moreover, $\Gamma_{\emph{out}}$ is nonempty.
		
		\item The function $f \in W^{1, \infty}(\R)$ with Lipschitz constant $L_f$ and is such that $f(0)=1$ and $f(1)= 0$ and $f(s) \ge 0 $ for every $0< s< 1$.
		
		\item The function  $g \in W^{1, \infty}(\R)$ with Lipschitz constant $L_g$, is monotonically decreasing and such that $g(0)=1$ and $g(1)= 0$.
		
		\item The functions $a \in L^q(\Omega; L^{\infty}((0,T) \times \Gamma_{\emph{in}}))$ and $L^q(\Omega; L^{\infty}((0,T) \times \Gamma_{\emph{out}}))$ for all $q \in [1, \infty)$ and there exist $a_0, b_0> 0$ such that
			\[
			P(\{\omega \in \Omega\colon a \ge a_0\})= P(\{\omega \in \Omega\colon b \ge b_0\})= 1.
			\]
%		Furthermore, $P(\{\omega \in \Omega: \, \min\{a, b\}< 1\})= 1$ and
%		it holds
%			\[
%			L_g \|a\|_{L^{\infty}(\Gamma_{\emph{in}})} \le \frac{1}{8\,c_{\rand}^2} \quad \text{as well as} \quad \|b\|_{L^{\infty}(\Gamma_{\emph{out}})} \le \frac{1}{8\,c_{\rand}^2}.
%			\]
		\item The potential satisfies $V \in L^q(\Omega; L^2((0,T); W^{1, \infty}(D)))$ for all $q \in [1, \infty)$ and $\Delta V = 0$, $P$-a.e..
	\end{enumerate}
	Then, there exists a solution $u \in L^2\big(0,T; H^1(D)\big) \cap L^2\big(0,T; H^1(D)^*\big)$ to problem~\eqref{random-prob2}, \(P\)\nobreakdash-a.s.
\end{theorem}

\begin{proof}
	The argument is the same as in Theorem~\ref{thm-random1} and thus we omit the proof.
\end{proof}

\subsection{Random solution}
We understand \emph{uniqueness} of solutions to~\eqref{random-prob1} in the almost sure sense: if $\{u(t, x, \omega)\}$ and $\{v(t, x,  \omega)\}$ are two solutions to problem~\eqref{random-prob1}, then
\[
	P\big(\{\omega \in \Omega\colon u(t, x, \omega)= v(t, x, \omega) \text{ for all } t \in  (0,T)\text{ and } x\in D\}\big)= 1.
\]
In this case $\{u(t, x, \omega)\}$ and $\{v(t, x, \omega)\}$ are \emph{indistinguishable} with respect to~\(P\). A similar argument holds true for the notion of uniqueness of the solution to~\eqref{random-prob2}.

In this setting, the solution
	\[
	\omega \mapsto u(\cdot \,, \cdot \,, \omega)= \mathcal S_1\big(\alpha(\cdot \,, \cdot \,, \omega), \beta(\cdot \,, \cdot \,, \omega), V(\cdot \,, \cdot \,, \omega), u^0(\cdot \,, \omega)\big)
	\]
to~\eqref{random-prob1} and the solution
	\[
	\omega \mapsto u(\cdot \,, \cdot \,, \omega)= \mathcal S_2\big(a(\cdot \,, \cdot\,, \omega), b(\cdot \,, \cdot \,,\omega), V(\cdot \,, \cdot \,,\omega), u^0(\cdot \,, \omega)\big)
	\]
to~\eqref{random-prob2} are random variables taking values in $L^2\big((0,T); H^1(\Om)\big) \cap L^2\big((0,T); H^1(\Om)^*\big)$.

Let us focus on problem~\eqref{random-prob1}. First, we equip the product space $L^{\infty}(D \times (0,T)) \times L^{\infty}(D \times (0,T)) \times L^\infty((0,T); W^{1, \infty}(D)) \times L^2(D)$ with the metric
\begin{align}\label{eq:Metric1}
	\MoveEqLeft[3] d_1\big((\alpha_1, \beta_1, V_1, u_1^0), (\alpha_2, \beta_2, V_2, u_2^0)\big)\coloneqq \\
	& \|\alpha_1- \alpha_2\|_{L^{\infty}((0,T) \times D)}+ \|\beta_1- \beta_2\|_{L^{\infty}((0,T) \times D)} \nonumber\\
	&\qquad+\|V_1- V_2\|_{L^{\infty}((0,T); W^{1, \infty}(D))}+ \|u_1^0- u_2^0\|_{L^2(D)}.\nonumber
\end{align}
From Theorem~\ref{thm:Lip1}, taking into account that the function $t \mapsto t^{\nicefrac12} $ is sublinear, we have%\todo{is $+\|u- u\|$ correct here, it was $\cdot$?}
\begin{align} 
	\MoveEqLeft[3] \|\mathcal S_1(\alpha_1, \beta_1, V_1, u_1^0) - \mathcal S_1(\alpha_2, \beta_2, V_2, u_2^0) \|_{L^2((0,T); H^1(D))} \\ \nonumber
	 & \le C e^{\tilde C\, T/2}\ \Bigl(\|\alpha_1- \alpha_2\|_{L^{\infty}((0,T) \times D)}+ \|\beta_1- \beta_2\|_{L^{\infty}((0,T) \times D)} \\ \nonumber
	 & \qquad\qquad\qquad   + \|V_1- V_2\|_{L^\infty((0,T); W^{1, \infty}(D))} +\|u_1^0- u_2^0\|_{L^2(D)}\Bigr)  \\
	 & \le  \overline C\cdot d_1\l((\alpha_1, \beta_1, V_1, u_1^0), (\alpha_2, \beta_2, V_2, u_2^0)\r), \label{eq:Lip1}
\end{align}
where $\overline C= C\, e^{\tilde C\, T/2}$.

Similarly, let us consider problem~\eqref{random-prob2}. We equip the product space $L^{\infty}((0,T) \times \Gammain) \times L^{\infty}((0,T) \times \Gammaout) \times L^2((0,T); W^{1, \infty}(\Om)) \times L^2(\Om)$ with the metric
\begin{align}\label{eq:Metric2}
	\MoveEqLeft[3] d_2\big((a_1, b_1, V_1, u_1^0), (a_2, b_2, V_2, u_2^0)\big)\coloneqq  \\
	 & \|a_1- a_2\|_{L^{\infty}((0,T) \times \Gammain)}+ \|b_1- b_2\|_{L^{\infty}((0,T) \times \Gammaout)}\nonumber  \\
	 & \qquad+\|V_1- V_2\|_{L^2((0,T); W^{1, \infty}(\Om))}+ \|\underline u_1^0- \underline u_2^0\|_{L^2(\Om)}. \nonumber
\end{align}
Reasoning as before, from Theorem~\ref{thm:Lip2}  we have that
\begin{align}
	\MoveEqLeft[2]\|\mathcal S_2(a_1, b_1, V_1, \underline u_1^0) - \mathcal S_2(a_2, b_2, V_2, \underline u_2^0)\|_{L^2((0,T); H^1(\Om))}\nonumber \\
	& \le   C e^{\tilde C\, T/2} \Big(\|a_1- a_2\|_{L^{\infty}((0,T) \times \Gammain)}+ \|b_1- b_2\|_{L^{\infty}((0,T) \times \Gammaout)} \\
	& \qquad\qquad\qquad + \|V_1- V_2\|_{L^2((0,T); W^{1, \infty}(\Om))} +  \|\underline u_1^0- \underline u_2^0\|_{L^2(\Om)} \Big) \nonumber\\
	& \le  C  e^{\tilde C\, T/2}  d_2\big((a_1, b_1, V_1, \underline u_1^0), (a_2, b_2, V_2, \underline u_2^0)\big).\label{eq:Lip2}
\end{align}
Summarizing, both problems are Lipschitz continuous with respect to the distances introduced. The Lipschitz constant, notably, strictly depends on the time horizon~$T$.

%				Uncertainty quantification
%	╰───────────────────────────────────
\section{Uncertainty quantification}
Having extended to results of Section \ref{sect3} to the random setting in the previous section, we can now tackle the main subject of this work, namely uncertainty quantification for time-dependent drift-diffusion equations.

We will investigate this time dependency by exchanging the measures. Exchanging the measure allows to replace a complicated probability measure by a simpler measure and execute computations for the simpler, perhaps discrete measure only. The continuity results will allow us to relate the original problem with the simpler, approximating problem and to connect them quantitatively.

In contrast to the  space dimensions, the dimension time indeed exhibits a new aspect and must be treated differently than spatial dimensions. Time progresses, revealing additional knowledge gradually, and adapted stochastic processes exactly address this aspect that past realizations are known, while random observations will be revealed in the unknown future.

Before addressing the novel temporal aspect in Section~\ref{sec:Process} below we introduce Wasserstein distances, which constitute the basis for the distance of stochastic processes.
\begin{definition}[Wasserstein distance, \cite{Villani2003}]
	Let $(\Xi, d)$ be Polish and \(P\) (\(\tilde P\), resp.)\ be a probability distribution on~\(\Xi\). The \emph{Wasserstein distance} of the measures~\(P\) and~\(\tilde P\) of order \(r\ge1\) is
	\begin{equation}\label{eq:Wasserstein}
		\W_r(P,\tilde P)^r\coloneqq \inf \iint_{\Xi\times \Xi} d(\xi,\tilde\xi)^r\,\pi(d\xi,d\tilde\xi),
	\end{equation}
	where the infimum is among all probability measures $\pi$ on $\Xi\times\Xi$ with marginals $P$ and $\tilde P$, i.e.,
	\begin{equation}\label{eq:WassersteinCond}
			\pi(A\times \Xi)=P(A)\ \text{ and }\ \pi(\Xi\times B)=\tilde P(B)
	\end{equation}
	for all measurable sets $A$, $B\in \mathcal B(\Xi)$.
\end{definition}
The most fundamental result on the Wasserstein distance is the Kantorovich--Rubinstein Theorem. The theorem states that
\begin{equation}\label{eq:18}
	\big|\E_P f- \E_{\tilde P}f\big|\le L_f\,\W_1(P,\tilde P),	
\end{equation}
where~\(f\colon\Xi\to\mathbb R\) is a \(\mathbb R\)-valued function with Lipschitz constant $L_f$, cf.~\eqref{eq:Lipschitz}; here,~\(\E_P\) (\(\E_{\tilde P}\), resp.)\ is the expectation with respect to the probability measure~\(P\) (\(\tilde P\), resp.), while~$\W_1$ comes from~\eqref{eq:Wasserstein} for $r= 1$.

\subsection{Quantities of interest}
As explained in Section~\ref{sec:QoI}, we are interested in quantities of interest of the random solutions. To involve the Kantorovich-Rubinstein theorem~\eqref{eq:18} we thus consider the concatenation of the QoI with the respective solution operator
\(\Phi\circ\mathcal S\) with $\mathcal S = \mathcal S_1$ and $\mathcal S = \mathcal S_2$, respectively.
%\jan{TODO: BLA BLA , ref QoI oben}
For random input, the compositions	% \[	\Phi\colon L^2((0,T); H^1(\Om)) \cap L^2((0,T); H^1(\Om)^*) \to \R.\]
\begin{align}
	\omega \mapsto \Phi\big(u(\cdot \,, \cdot \,, \omega)\big)= (\Phi \circ \mathcal S_1)\l(\alpha(\cdot \,, \cdot \,, \omega), \beta(\cdot \,, \cdot \,, \omega), V(\cdot \,, \cdot \,, \omega), u^0(\cdot \,, \omega)\r)
\shortintertext{and}
	\omega \mapsto \Phi\big(u(\cdot \,, \cdot \,, \omega)\big)= (\Phi \circ \mathcal S_2)\l(a(\cdot \,, \cdot \,, \omega), b(\cdot \,, \cdot \,, \omega), V(\cdot \,, \cdot \,, \omega), \underline u^0(\cdot \,, \omega)\r)	\end{align}
are $\mathbb R$\nobreakdash-valued random variables. 

Assembling the results from the preceding section we have the following preliminary result.
\begin{theorem}\label{thm:Wasserstein}
	Let~\(P\) and~\(\tilde P\) be probability distributions on \(\Xi\) (cf.~\eqref{eq:Xi}), equipped with the metric~\eqref{eq:Metric1} or~\eqref{eq:Metric2}. Then the quantities of interest of the problems~\eqref{problem1} and~\eqref{problem2} satisfy
	\[	\left|\E_P \Phi\big(\mathcal S\big)- \E_{\tilde P} \Phi\big(\mathcal S\big)\right|		\le L_\Phi\  C   e^{\tilde C\, T/2} \, \W_r(P, \tilde P) \]
	both for $\mathcal S = \mathcal S_1$ and $\mathcal S = \mathcal S_2$, where $L_\Phi$ is the Lipschitz constant of the quantity of interest and $r\ge1$.
\end{theorem}
\begin{proof}
	The Lipschitz continuity for~\(\mathcal S_1\) follows from Theorem~\ref{thm:Lip1} (cf.\ Remark~\ref{rem:1}) and the Lipschitz continuity of~\(\mathcal S_2\) from Theorem~\ref{thm:Lip2} (cf.\ Remark~\ref{rem:2}). 
	The Lipschitz constant of the composition $\Phi\circ \mathcal S$ is 
	\[L= L_\Phi \, L_{\mathcal S}.\]
	The result for $r=1$ follows thus follows from the Kantorovich--Rubinstein theorem~\eqref{eq:18} above. The result for $r\ge1$ follows from H$\ddot{\text{o}}$lder's inequality, as $\W_1(P,\tilde P)\le \W_r(P,\tilde P)$.
\end{proof}

%					stochastic process
%----------------------------------------
\subsection{The novel temporal aspect and stochastic process}\label{sec:Process}
The boundary conditions $\alpha$ and $\beta$ of the differential equation~\eqref{problem1} ($a$ and $b$ for  problem~\eqref{problem2}, respectively)\ are random functions, chosen according the probability measure~$P$ as \(\omega\mapsto \alpha(t, x, \omega).\) Given the parabolic equation and its solution it is sufficient to know $\alpha(t^\prime, x)$ and $\beta(t^\prime, x)$ for $t^\prime\le t$ to find $u(t, x)$. Put differently, $\alpha(t'', x)$ for all $t''>t$ does not impact $u(t, x)$.

Further, the solution operators~\eqref{solution1} and~\eqref{solution2} of the differential equations~\eqref{problem1} and~\eqref{problem2} feature a semigroup property, i.e., solving the equations from $t=0$ up to time ~$t=T$, say, and then solving again from $T$ with the solution obtained for the next time span up to ~$T^\prime > T$ is the same as solving the initial problem up to~$T^\prime$. We shall exploit this property now for the stochastic problem setting.

To this end consider the functions $\alpha$, $\beta$ (for the first problem) and the functions $a$ and $b$ of the second problem. These functions evolve over time, and they are observed over time. At time~$t$ their past is known but, in a random environment, not their future. With these functions we associate the~\(\sigma\)\nobreakdash-algebras
\begin{equation}\label{eq:20}
	\mathcal F_t\coloneqq \sigma \left(\left.
		\begin{array}{l}
			\{\omega\in\Omega\colon\alpha(t^\prime, x, \omega)\in A\}, \\
			\{\omega\in\Omega\colon\beta(t^\prime,x,\omega)\in A\}, \\
			\{\omega\in\Omega\colon V(t^\prime,x,\omega)\in A\}
		\end{array}\right| t^\prime\le t, x\in D\text{ and }A\in\mathcal B(\Xi)\right)
\end{equation}
generated by the sets of parametric functions, which cannot be distinguished further up to time~$t$;
this $\sigma$\nobreakdash-algebra $\mathcal F_t$ models the information available at time~$t$. The family
\[\mathcal F\coloneqq \big(\mathcal F_t\big)_{t\in [0,T]}\] of increasing $\sigma$\nobreakdash-algebras constitutes a filtration.

For the second problem, respectively, the filtration is built form the $\sigma$\nobreakdash-algebras
\begin{equation}\label{eq:21}
	\mathcal F_t\coloneqq \sigma\left(\left.
		\begin{array}{l}
			\{\omega\in\Omega\colon a(t^\prime,x,\omega)\in A\}, \\
			\{\omega\in\Omega\colon b(t^\prime,x,\omega)\in A\}, \\
			\{\omega\in\Omega\colon V(t^\prime,x,\omega)\in A\}
		\end{array}\right| t^\prime\le t, x\in D\text{ and }A\in\mathcal B(\Xi) \right).
\end{equation}

\begin{remark}
	Note that the $\sigma$\nobreakdash-algebra $\mathcal F_t$ in~\eqref{eq:20} (\eqref{eq:21}, resp.)\ are generated by the random variables $\omega\mapsto \alpha(t^\prime,x,\omega)$, that is,
	\[\mathcal F_t= \sigma\big(\alpha(t^\prime, x)\colon x\in D\text{ and }t^\prime\le t\big).\]
\end{remark}

With this, the Wasserstein distance generalizes as follows by involving the filtrations.
\begin{definition}[Nested distance, aka.\ process distance, cf.\ \cite{Pflug2009, PflugPichlerBuch}]
	Let~\(P\) and~\(\tilde P\) be probability measures and \(\mathcal F=(\mathcal F_t)_{t\in[0,T]}\) and \(\tilde{\mathcal F}=(\tilde{\mathcal F}_t)_{t\in[0,T]}\) be filtrations. The nested distance is%\jan{Different name?}	
	\begin{equation}\label{eq:Nested}
		\mathrm{nd}_r(P,\tilde P)^r\coloneqq \iint_{\Xi\times\Xi} d(\xi,\tilde\xi)^r\,\pi(d\xi, d\tilde\xi),\end{equation}
	where the infimum is among all probability measures $\pi$ on $\Xi\times\Xi$ with \emph{conditional} marginals $P$ and $\tilde P$, i.e.,
	\begin{align}\label{eq:NestedCond}
		\pi(A\times \Xi\mid \mathcal F_t)&=P(A\mid\mathcal F_t)\quad\text{and} \\
		\pi(\Xi\times B\mid\mathcal F_t)&=\tilde P(B\mid\tilde{\mathcal F}_t)
	\end{align}
	for all measurable sets $A$, $B\in \mathcal B(\Xi)$.
\end{definition}
The notation in~\eqref{eq:Nested} is an abbreviation, as the nested distance depends on the measure~$P$ \emph{and} the filtration~$\mathcal F$. Our notation suppresses this dependence on~$t$, the filtration here captures time for both, $\mathcal F$ and $\tilde{\mathcal F}$.

The measure $\pi(A\times B) = P(A)\tilde P(B)$ satisfies the constraints~\eqref{eq:NestedCond}, hence the nested distance is well-defined. Further,~\eqref{eq:NestedCond} includes the constraints~\eqref{eq:WassersteinCond} so that \(d_r\le \mathrm{nd}_r\).
It follows from Theorem~\ref{thm:Wasserstein} that
$P\mapsto \E_P\Phi\big(\mathcal S\big)$ is continuous with respect to the nested distance.

\subsection{Transitory distributions}
The solution operators~$\mathcal S_1$ and~$\mathcal S_2$ are continuous in the space $L^2\big((0,T); H^1(D)\big)$, where its metric~\eqref{eq:Metric1} and~\eqref{eq:Metric2} can be temporarily decomposed as
\begin{equation}\label{eq:42}
	\|u\|^2_{L^2((0,T); H^1(D))}
	%=\int_0^T \|u(s,\cdot)\|_{H^1(D)}^2\,ds
	= \int_0^t \|u(s,\cdot)\|_{H^1(D)}^2\,ds + \int_t^T \|u(s,\cdot)\|_{H^1(D)}^2\,ds
\end{equation}
for $t\in(0,T)$. %This dependency on time is made explicit by writing \[	\|u\|_{t_1:t_2}\coloneqq \int_{t_1}^{t_2} \|u(s,\cdot)\|_{H^1(D)}^2\,ds.\]

For the following extension we shall denote by $P_t$ the restriction of $P$ to $\mathcal F_t$ and also employ the probability kernels 
\begin{equation}\label{eq:44}
	P(\cdot\mid\mathcal F_t)= P^{|\mathcal F_t}(\cdot);
\end{equation}
these disintegrated measures exist by \cite[Chapter~5]{Kallenberg2002Foundations}.  %\jan{Genauere Refernz}

Concatenated, these kernels~\eqref{eq:44} are
\begin{equation}\label{eq:41}
	P(A_{t_1}\times\dots \times A_T)
	=\int_{A_{t_1}}\dots\int_{A_T} P_T(d\xi_T\mid \xi_1\dots\xi_{t_{n-1}})\cdots P_{t_2}(d\xi_{t_2}\mid \xi_{t_1})\,P_{t_1}(d\xi_{t_1}),
\end{equation}
the initial probability measure~$P$, where $0=t_0< t_1 < \dots < t_n=T$.

\begin{theorem}[Gluing theorem]\label{thm:gluing}
	Suppose that the conditional Wasserstein distances are H\"{o}lder continuous in time with
	\begin{equation}\label{eq:40}
		d_r \big( P_t^{|\mathcal F_{t^\prime}}, \tilde P_t^{|\tilde{\mathcal F}_{t^\prime}}\big)^r \le C_H|t^\prime-t| \quad \text{ a.s.\ for all }t^\prime <t\le T,
	\end{equation} where $C_H>0$.
	Then we have 
	\begin{equation}\label{eq:43}
		\left|\E_P \Phi(\mathcal S)-\E_{\tilde P}\Phi(\mathcal S)\right|
		\le L_\Phi\,C\, \max\left(1; \sqrt{C_H\, T}\right) e^{\tilde C\,T/2}\ \mathrm{nd}_2(P, \tilde P),
	\end{equation}
	where $\mathrm{nd}_2$ is the nested distance with rate $r=2$.
\end{theorem}
\begin{proof}
	The conditional probabilities of~$P$ and~$\tilde P$ in~\eqref{eq:40} can be concatenated as in~\eqref{eq:41} to the bivariate probability measure (cf.\ \cite{PichlerSchlotterMartingales})
	\[	\pi (A\times B)=\iint_{A_{t_1}\times B_{t_1}}\dots\iint_{A_T\times B_T} 
		\pi_T(d\xi_T,d\eta_T\mid \xi_{t_1},\eta_{t_1}\dots\xi_{t_{n-1}},\eta_{t_{n-1}})\cdots
		\pi_{t_1}(d\xi_{t_1}, d\eta_{t_1}), \]
	where $A=A_{t_1}\times\dots\times A_T$ and $B=B_{t_1}\times\dots\times B_T$ and each measure
	\[	\pi_t(\cdot\times\cdot \mid \xi_{t_1},\eta_{t_1}\dots\xi_{t_i},\eta_{t_i})\] 
	solves the Wasserstein problem~\eqref{eq:Wasserstein} for the conditional measures $P_t\big(\cdot\mid \xi_{t_1}\dots \xi_{t_i}\big)$ and 	$\tilde P_t\big(\cdot\mid \eta_{t_1}\dots \eta_{t_i}\big)$, respectively. 
	As a consequence, the measure~$\pi$ also satisfies the marginal constraints~\eqref{eq:NestedCond} and is thus feasible for the nested distance.

	Now recall that each distance~\eqref{eq:Metric1} and~\eqref{eq:Metric2} consists of three ingredients: 
	\begin{enumerate}[noitemsep, nolistsep]
		\item\label{enu:1} The initial conditions $u_0$ and $\underline u^0$ are measured in $L^2(D)$, which does not depend on time;
		\item\label{enu:2} the (respective) potential $V$ is measured in the space $L^{\infty}\big((0,T); W^{1,\infty}(D)\big)$ and $L^2\big((0,T), W^{1,\infty}(D)\big)$, while 
		\item\label{enu:3} the remaining terms are measured in the spaces $L^\infty\big((0,T)\times D\big)$, $L^{\infty}\big((0,T) \times \Gammain \big)$, and $L^{\infty}\big((0,T) \times \Gammaout \big)$, respectively. %, each with its associated state space.
	\end{enumerate}

	For this reason we need to separate the results as well.

	The initial condition~\ref{enu:1} does not need to be addressed separately, this term is independent of time.
	For $u\in L^2((0,T); H^1(D))$ (see~\ref{enu:2}) recall the composition~\eqref{eq:42} of the norm, which we may intersect according the tessellation as
	\begin{equation*}
		\|\xi_s\|^2_{L^2((0,T);H^1(D))}
		= \int_0^T \|\xi_s\|_{H^1(D)}^2\,ds
		= \sum_{i=1}^n \int_{t_{i-1}}^{t_i} \|\xi_s\|_{H^1(D)}^2\,ds.
	\end{equation*}
	Integrating with respect to $\pi$ during constant times $s\in (t_{i-1},t_i)$ as the nested distance~\eqref{eq:Nested} gives 
	\begin{align}
		\iint_{\Xi\times\Xi}& 
		d(\xi_{t_1},\dots,\xi_{t_n}, \tilde\xi_{t_1},\dots\tilde\xi_{t_n})\, 
		\pi_T(d\xi_T,d\eta_T\mid \xi_{t_1},\eta_{t_1}\dots\xi_{t_{n-1}},\eta_{t_{n-1}})\cdots
		\pi_{t_1}(d\xi_{t_1}, d\eta_{t_1})\\
		&= \iint \bigg(\int_{t_0}^{t_1}\|\xi_{s_1}-\eta_{s_1}\|^2\,ds_1 +\dots\\
		&\qquad\qquad +\Big(\int_{t_{n-1}}^{t_n}\|\xi_{s_n}-\eta_{s_n}\|^2ds_n\,\pi_{t_n}(d\xi_{t_n},\eta_{t_n}\mid \xi_{t_1},\eta_{t_1},\dots\xi_{t_{n-1}},\eta_{t_{n-1}})\Big)\label{eq:13}\\
		&\qquad\qquad\qquad\qquad\qquad\qquad\qquad\qquad \dots\bigg)\,\pi_{t_1}(d\xi_{t_1},d\eta_{t_1}).
	\end{align}
	The conditional measures in~\eqref{eq:13} are bounded by~\eqref{eq:40} and so are the others for all~$t_i$, $i=1,\dots n$. They accumulate to 
	\[	\mathrm{nd}_2(P,\tilde P)^2 \le C_H\sum_{i=1}^n \left|t_i-t_{i-1}\right|=C_H\, T,\]
	revealing thus the term $\sqrt{C_H\,T}$ in~\eqref{eq:43}.

	As for~\ref{enu:3}, it holds that
	\begin{equation*}
		\|\xi_s\|_{L^\infty((0,T); H^1(D))}
			= \max_{i=1,\dots,n} \sup_{s\in(t_{i-1},t_i)} \|\xi_s\|_{H^1(D)}^2 .
	\end{equation*}
	With that, the relation~\eqref{eq:13} reads
	\begin{align}
		\iint_{\Xi\times\Xi}& 
		d(\xi_{t_1},\dots,\xi_{t_n}, \tilde\xi_{t_1},\dots\tilde\xi_{t_n})\, 
		\pi_T(d\xi_T,d\eta_T\mid \xi_{t_1},\eta_{t_1}\dots\xi_{t_{n-1}},\eta_{t_{n-1}})\cdots
		\pi_{t_1}(d\xi_{t_1}, d\eta_{t_1})\\
		&= \iint \max_{i=1,\dots,n}\bigg(\sup_{s_1\in(t_0,t_1)}\|\xi_{s_1}-\eta_{s_1}\|, \dots\\
		&\qquad\qquad \Big(\sup_{s_n\in(t_{n-1},t_n)}\|\xi_{s_n}-\eta_{s_n}\|\,\pi_{t_n}(d\xi_{t_n},\eta_{t_n}\mid \xi_{t_1},\eta_{t_1},\dots\xi_{t_{n-1}},\eta_{t_{n-1}})\Big)\label{eq:14}\\
		&\qquad\qquad\qquad\qquad\qquad\qquad\qquad\qquad \dots\bigg)\,\pi_{t_1}(d\xi_{t_1},d\eta_{t_1}).
	\end{align}
	Now recall the continuity properties~\eqref{eq:Lip1} and~\eqref{eq:Lip2} of the operator~$\mathcal S_1$ ($\mathcal S_2$, resp.)\ and the Lipschitz constant $L_\Phi$ of~$\Phi$.
	Passing to the infimum of all measures, the assertion~\eqref{eq:43} follows.
\end{proof}

The following corollary to Theorem~\ref{thm:gluing} above covers the remaining time from~$t$ up to~$T$. The constants reveal that the estimates improve gradually with decreasing remaining time. The basic order, for small time intervals, is $\sqrt \tau$, where $\tau=T-t$ is the remaining time.
\begin{corollary}
	Given the conditions of Theorem~\ref{thm:gluing} it holds that
	\begin{equation}
		\left|\E_{P^{|\mathcal F_t}} \Phi(\mathcal S)-\E_{{\tilde P}^{|\mathcal F_t}}\Phi(\mathcal S)\right|
		\le L_\Phi\,C\, \max\left(1; \sqrt{C_H(T-t)}\right) e^{\tilde C(T-t)/2}\ \mathrm{nd}_2\left(P^{|\mathcal F_t}, {\tilde P}^{|\mathcal F_t}\right) \text{ a.s.},
	\end{equation}
	where $\mathrm{nd}_2$ is the nested distance with rate $r=2$.
\end{corollary}

\begin{remark}[Markovian]
	We want to emphasize that the situation captured in the filtrations~\eqref{eq:20} and~\eqref{eq:21} correspond to the \emph{non-Markovian} situation, that is, conditioning the stochastic processes on~$\mathcal F_t$ reflects their entire history, not just their state at the final time~$t$.
\end{remark}

\begin{remark}[The interim perspective]
	The fundamental relation~\eqref{eq:43} reveals continuity from the starting time $t=0$ up to the end time $t=T$. Theorem~\ref{thm:gluing} generalizes to any interim perspective, ranging from $t$ to $t^\prime$ provided that $0\le t\le t^\prime\le T$ while preserving the continuity property~\eqref{eq:43}.

	In this way it is possible to predict the aberrations in the expected value to be expected in the remaining time interval.
\end{remark}

%			Summary
%	╰───────────────────────────────────
\section{Summary}\label{sec:summary}
This paper exposes drift-diffusion equations to a random environment, where parameters of the governing equations are chosen at random. In contrast to the usual theory of random coefficients, the parameter driving the solution of the system are revealed gradually, over time. That is, partial observations are known (they are deterministic), whereas only future observations are random. Over time, information increases, whereas the random variability shrinks.

This paper captures evolving information by the~$\sigma$\nobreakdash-algebras and employs a distance for stochastic processes to assess the situation. We present strict continuity results when comparing the outcome with realization from a different stochastic process.

%               Bibliography
%	╰───────────────────────────────
%\bibliography{Literature}

\begin{thebibliography}{30}

\bibitem{Acciaio2016}{B. Acciaio, J.  Backhoff Veraguas,  and A. Zalashko, \emph{Causal optimal transport and its links to enlargement of filtrations and continuous-time stochastic optimization}, Stochastic Processes and their Applications, {\bf 130} (2016), no. 5, 2918--2953.}


 \bibitem{Adams}{R.A. Adams, ``Sobolev Spaces'', Academic Press, New York-London, 1975.}

\bibitem{BackhoffBeiglbock}{J.D. Backhoff Veraguas, M. Beiglb\"{o}ck, M. Eder,  and A. Pichler, \emph{Fundamental properties of process distances}, Stochastic Processes and their Applications, (2020).}


 \bibitem{BHP}{M. Burger, I. Humpert, and J.-F. Pietschmann, \emph{On Fokker-Planck equations with in- and outflow mass}, Kinetic {\&} Related Models, {\bf 13} (2020), no. 2, 249--277.}

\bibitem{Egger}{H. Egger,  J.-F. Pietschmann,  and M. Schlottbom, \emph{Identification of Chemotaxis Models with Volume-Filling}, SIAM Journal on Applied Mathematics, {\bf 75} (2015), no. 2, 275--288.}


 \bibitem{Evans1998}{L.C. Evans, ``Partial Differential Equations'', American Mathematical Society, 1998.}

 \bibitem{ErnstPichlerSprungk2020}{O.G. Ernst, A. Pichler, and B. Sprungk, \emph{Sensitivity of uncertainty propagation for the elliptic diffusion equation}, https://arxiv.org/abs/2003.03129.}

 \bibitem{GSW}{S.N. Gomes, A.M. Stuart, and M.-T. Wolfram, \emph{Parameter estimation for macroscopic pedestrian dynamics models from microscopic data}, SIAM J. Appl. Math {\bf 79} (2019), no. 4, 1475--1500.}
 
 \bibitem{Ina2021_PhD}{I. Humpert, ``Mathematical Models of Transport Phenomena with In- and Outflow'', PhD Thesis, WWU M\"{u}nster, 2021.}



\bibitem{humpert_role_2019}{I. Humpert, D. Di Meo, A. P\"uschel,  and J.-F. Pietschmann, \emph{On the Role of Vesicle Transport in Neurite Growth: Modelling and Experiments}, (2019) http://arxiv.org/pdf/1908.02055v1.}

\bibitem{Ansgar_bounded}{A. J{\"u}ngel, \emph{The boundedness-by-entropy method for cross-diffusion systems}, {Nonlinearity} {\bf 28} (2015), no. 6, 1963.}

\bibitem{Kallenberg2002Foundations}{O. Kallenberg, ``Foundations of Modern Probability'', Springer, New York, 2002.}


\bibitem{Karatzas1988}{I. Karatzas  and S.E. Shreve,  ``Brownian Motion and Stochastic Calculus'', Graduate Texts in Mathematics, Springer-Verlag New York, 1991.}


 \bibitem{MW}{G. Marino and P. Winkert, \emph{Moser iteration applied to elliptic equations with critical growth on the boundary}, Nonlinear Anal. {\bf 180} (2019), 154--169.}
 
 \bibitem{Pflug2009}{G.C. Pflug, \emph{Version-Independence and nested distributions in multistage stochastic optimization}, SIAM Journal on Optimization, {\bf 20} (2009), 1406--1420.}
 
\bibitem{PflugPichler2011}{G.C. Pflug and A. Pichler, \emph{A Distance for Multistage Stochastic Optimization Models}, SIAM Journal on Optimization, {\bf 22} (2012), 1--23.}

\bibitem{PflugPichlerBuch}{G.C. Pflug and A. Pichler, ``Multistage Stochastic Optimization'', Springer Series in Operations Research and Financial Engineering, Springer, 2014.}

\bibitem{PichlerSchlotterMartingales}{A. Pichler  and R. Schlotter, \emph{Martingale Characterizations of Risk-Averse Stochastic Optimization Problems}, Mathematical Programming, {\bf 181} (2019), no. 2, 377--403.}

 
 \bibitem{RachevRueschendorf}{S. Rachev  and L. R\"{u}schendorf, ''Mass Transportation Problems'' Volume I: Theory, Volume II: Applications, Springer, New York, 1998.}
 
 \bibitem{Simon}{J. Simon, \emph{Compact sets in the space ${L}^p (0,{T}; {B})$}, Annali di Matematica Pura ed Applicata, {\bf 146} (1986), no. 1, 65--96.}
 
\bibitem{Villani2003}{C. Villani, ``Topics in Optimal Transportation'', Graduate Studies in Mathematics, American Mathematical Society, 2003.}




 \end{thebibliography}

\appendix

\section{Appendix}\label{sec:appendix}

The objective of this section is to prove Theorem \ref{rem:existence1}. We will divide the proof in two lemmas.

\begin{lemma}
\label{lemma:exist}

Let hypotheses \ref{H1}--\ref{H4} be satisfied. Then, there exists a unique  $u \in L^p((0,T); W^{2,p}(D)) \cap W^{1,p}((0,T), L^p(D))$ that satisfies \eqref{weak-sol1}.

\end{lemma}

\begin{proof}

We want to apply the Banach's fixed point theorem, following the ideas of \cite[Theorem 3.1]{Egger}.  We consider the nonempty, closed set
	\[
	\mathcal M= \{u \in L^{\infty}((0,T); L^2(D)): \, \|u\|_{L^{\infty}((0,T); L^2(D))} \le C_{\mathcal M}\},
	\]
with $T, C_{\mathcal M}>0$ to be specified.  Then we define the mapping 
	\[
	\Phi\colon \mathcal M \to L^{\infty}((0,T); L^2(D)), \quad \Phi(\tilde u)= u,
	\]
where $u$ is the weak solution to the linearized problem
	\begin{equation}
	\label{linearized}
	\begin{aligned}
	& \partial_t u+ \divergenz(-\nabla u+ f(\tilde u) \nabla V)= \alpha(t, x) f(\tilde u)- \beta(t, x) u \quad && \text{in } (0,T) \times D, \\
	& (-\nabla u+ f(\tilde u) \nabla V) \cdot n= 0 && \text{on } (0,T) \times \rand.
	\end{aligned}
	\end{equation}
We first show that $\Phi$ is self-mapping. Indeed, testing \eqref{linearized} with $\varphi= u$, integrating over $D$ and rearranging give
	\[
	\frac{d}{dt} \into u^2 dx+ \into |\nabla u|^2 dx  \le \into f(\tilde u)^2 |\nabla V|^2 dx+ \into \alpha^2 f(\tilde u)^2 dx+ \into u^2 dx.
	\]
A Gronwall argument then gives
	\[
	\begin{split}
	& \|u(t, \cdot)\|_{L^2(D)}^2+ \int_0^t |\nabla u|^2 dx \\
	& \le e^T \l(\int_0^T \into f(\tilde u)^2 |\nabla V|^2 dx dt+ \int_0^T \into \alpha^2 f(\tilde u)^2 dx dt+ \|u^0\|_{L^2(D)}^2 \r) \\
	& \le e^T |D| \l(\|f(\tilde u)\|_{L^{\infty}(\R)}^2 \|V\|_{L^{\infty}((0,T); W^{1, \infty}(D)}^2+ \|\alpha\|_{L^{\infty}((0,T) \times D)}^2 \|f(\tilde u)\|_{L^{\infty}(\R)}^2+ 1 \r),
	\end{split}
	\]
for every $t \in (0,T)$, which implies that
	\[
	\sup_{t \in (0,T)} \|u(t, \cdot)\|_{L^2(D)}^2 \le C,
	\]
where $C= C(T, D, \alpha, V, \|f\|_{L^{\infty}(D)})$ but doesn't depend on $u$. This shows that $\Phi$ is self-mapping. 

We next verify that $\Phi$ is actually a contraction. To this end, let $u_1= \Phi(\tilde u_1)$ and $u_2= \Phi(\tilde u_2)$ be two solutions to \eqref{linearized} for $\tilde u_1, \tilde u_2 \in \mathcal M$. We then consider the difference of the two equations and test the corresponding equation with $\varphi= u_1- u_2$. Using the fact that $\displaystyle -\into \beta (u_1- u_2)^2 dx \le 0 $ and the Lipschitz continuity of $f$ we achieve
	\[
	\begin{split}
	& \frac{d}{dt} \into (u_1- u_2)^2 dx+ \into |\nabla(u_1- u_2)|^2 dx \\
	& \le \into (f(\tilde u_1)- f(\tilde u_2))^2 |\nabla V|^2 dx+ \into \alpha^2 (f(\tilde u_1)- f(\tilde u_2))^2 dx+ \into (u_1- u_2)^2 dx \\
	& \le L_f^2 \into (\tilde u_1- \tilde u_2)^2 |\nabla V|^2 dx+ L_f^2 \into \alpha^2 (\tilde u_1- \tilde u_2)^2 dx+ \into (u_1- u_2)^2 dx.
	\end{split}
	\]
Again a Gronwall argument gives
	\[
	\begin{split}
	& \|(u_1- u_2)(t, \cdot)\|_{L^2(D)}^2+ \int_0^t \|\nabla (u_1- u_2)\|_{L^2(D)}^2 \\
	& \le L_f^2 e^T \l(\int_0^T \into (\tilde u_1- \tilde u_2)^2 |\nabla V|^2 dx dt+ \int_0^T \into \alpha^2 (\tilde u_1- \tilde u_2)^2 dx dt \r) \\
	& \le L_f^2 T e^T  \l(\|V\|_{L^{\infty}((0,T); W^{1, \infty}(D)}^2+ \|\alpha\|_{L^{\infty}((0,T) \times D)}^2 \r) \|\tilde u_1- \tilde u_2 \|_{L^{\infty}((0,T); L^2(D)}^2,
	\end{split}
	\]
which implies that 
	\[
	\|u_1- u_2\|_{L^{\infty}((0,T); L^2(D))}^2 \le L_f^2 T e^T \l(\|V\|_{L^{\infty}((0,T); W^{1, \infty}(D)}^2+ \|\alpha\|_{L^{\infty}((0,T) \times D)}^2 \r) \|\tilde u_1- \tilde u_2\|_{L^{\infty}((0,T); L^2(D))}^2.
	\]
Choosing $T>0$ small enough so that
	\[
	L_f^2 T e^T  \l(\|V\|_{L^{\infty}((0,T); W^{1, \infty}(D)}^2+ \|\alpha\|_{L^{\infty}((0,T) \times D)}^2 \r)< 1
	\]
shows that $\Phi$ is a contraction. Then the Banach's fixed point theorem applies and we infer the existence of a unique $u \in \mathcal M$ such that $\Phi(u)= u$. Then the conclusion follows by applying a standard regularity theory.
\end{proof}

In the next lemma we show that the solution to \eqref{weak-sol1} satisfies the box constraint $0 \le u \le 1$.

\begin{lemma}
\label{lemma:box}

Let hypotheses \ref{H1}--\ref{H4} be satisfied and let  $u \in L^p((0,T); W^{2,p}(D)) \cap \\ W^{1,p}((0,T), L^p(D))$ be the unique solution to \eqref{weak-sol1}. Then, it holds $0 \le u \le 1$.

\end{lemma}

\begin{proof}

We follow \cite[Lemma 3.2]{Egger}. For every $\eps>0$ we consider the function $\eta_{\eps} \in W^{2, \infty}(\R)$ given by
	\[
	\eta_{\eps}(u)=
	\begin{cases}
	0 & \text{if } u \le 0 \\
	\displaystyle \frac{u^2}{4 \eps} & \text{if } 0< u \le 2\eps \\
	u- \eps & \text{if } u> 2\eps.
	\end{cases}
	\]
We observe that $\eta_{\eps}$ is a regularization of the function $u^+= \max\{u, 0\}$. Further, it is easily seen that
	\[
	\eta_{\eps}'(u)=
	\begin{cases}
	0 & \text{if } u \le 0 \\
	\displaystyle \frac{u}{2 \eps} & \text{if } 0< u \le 2\eps \\
	1 & \text{if } u> 2\eps.
	\end{cases}
	\quad \text{as well as} \quad
	\eta_{\eps}''(u)=
	\begin{cases}
	0 & \text{if } u \le 0 \\
	\displaystyle \frac{1}{2 \eps} & \text{if } 0< u \le 2\eps \\
	0 & \text{if } u> 2\eps.
	\end{cases}
	\]
 We want to show that $u \le 1$. Using \eqref{weak-sol1}, an integration by parts and the Young's inequality gives
	\begin{equation}
	\label{box1}
	\begin{split}
	& \frac{d}{dt} \into \eta_{\eps}(u-1) dx= \into \eta_{\eps}'(u-1) \partial_t u \; dx \\
	&= \into -\eta'_{\eps}(u-1) \nabla \cdot (-\nabla u+ f(u) \nabla V)+ \eta'_{\eps}(u-1) \alpha f(u)- \eta'_{\eps}(u-1) \beta u \; dx \\
	&= \into -\eta''_{\eps}(u-1) |\nabla u|^2+ \eta''_{\eps}(u-1) \nabla u \cdot \nabla V f(u)+ \eta'_{\eps}(u-1) \alpha f(u)- \eta'_{\eps}(u-1) \beta u \; dx \\
	& \le -\frac{1}{2} \into \eta''_{\eps}(u-1) |\nabla u|^2 dx+ \frac{1}{2} \into \eta''_{\eps}(u-1) f(u)^2 |\nabla V|^2 dx+ \into \eta'_{\eps}(u-1) \alpha f(u) dx\\
	& \qquad - \into \eta'_{\eps}(u-1) \beta u \; dx.
	\end{split}
	\end{equation}
We claim that the last three integrals on the right-hand side of \eqref{box1} vanish when $\eps \to 0$. Indeed, we consider the set
	\[
	\Om_{\eps}\coloneqq \{x \in \Om: \, 1 \le u(t, x) \le 1+ 2\eps\}
	\]
and use the fact that $f(1)= 0$ to have
\begin{equation}
	\label{box2}
	\begin{split}
	& \int_{D_{\eps}} \eta_{\eps}''(u-1) f(u)^2 |\nabla V|^2 dx= \int_{D_{\eps}} \eta_{\eps}''(u-1) (f(u)- f(1))^2 |\nabla V|^2 dx \\
	&= \int_{D_{\eps}} \eta''_{\eps}(u-1) f'(\xi) (u-1)^2 |\nabla V|^2 dx \le \|f'\|_{L^{\infty}(\R)}^2 \int_{D_{\eps}} \frac{1}{2 \eps} (2\eps)^2 |\nabla V|^2 dx\\
	& \le 2\eps \|f'\|_{L^{\infty}(\R)}^2 \|\nabla V\|_{L^2(D)}^2 \to 0,
	\end{split}
\end{equation}
as well as
\begin{equation}
	\label{box3}
	\int_{D_{\eps}} \eta'(u-1) \alpha f(u) dx \le 2\eps \|\alpha\|_{L^{\infty}(D)} \|f'\|_{L^{\infty}(D)} |D_{\eps}| \to 0,
\end{equation}
and finally
\begin{equation}\label{box4}
	\int_{D_{\eps}} \eta'_{\eps}(u-1) \beta u \; dx \le 2\eps \beta_{L^{\infty}(D)} |D_{\eps}| \to 0,
\end{equation}
as $\eps \to 0$. Taking into account \eqref{box2}, \eqref{box3}, and \eqref{box4}, from \eqref{box1} we have
	\[
	\frac{d}{dt} \into (u-1)^+ dx= \lim_{\eps \to 0} \into \eta_{\eps}(u-1) \; dx \le 0
	\]
which implies that
	\[
	\into (u-1)^+ dx \le \into (u^0- 1)^+ dx.
	\]
Taking hypothesis \ref{H1} into account, it must be $(u-1)^+= 0$, from which $u \le 1$. 

For the inequality $u \ge 0$ we use a similar argument, considering $u^-$ instead of $(u-1)^+$ and using $f(0)= 0$ instead of $f(1)= 0$.
\end{proof}

\end{document}